\documentclass[preprint,11pt]{imsart}

\RequirePackage[OT1]{fontenc}
\RequirePackage{amsthm,amsmath}
\RequirePackage[numbers]{natbib}
\RequirePackage[colorlinks,citecolor=blue,urlcolor=blue]{hyperref}

\usepackage{amssymb}

\startlocaldefs
\numberwithin{equation}{section}
\theoremstyle{plain}
\newtheorem{theorem}{Theorem}[section]

\newtheorem{lemma}{Lemma}[section]

\newtheorem{conjecture}{Conjecture}[section]

\theoremstyle{definition}
\newtheorem{remark}[theorem]{Remark}
\endlocaldefs

\newcommand{\bea}{\begin{eqnarray}}
\newcommand{\ena}{\end{eqnarray}}
\newcommand{\beq}{\begin{equation}}
\newcommand{\enq}{\end{equation}}
\newcommand{\beas}{\begin{eqnarray*}}
\newcommand{\enas}{\end{eqnarray*}}

\begin{document}

\begin{frontmatter}
\title{The rate of convergence of some asymptotically chi-square distributed statistics by Stein's method}
\runtitle{Approximation of chi-square statistics by Stein's method}

\begin{aug}
\author{\fnms{Robert E.} \snm{Gaunt}}
\and
\author{\fnms{Gesine} \snm{Reinert}}

\runauthor{Gaunt and Reinert}

\affiliation{University of Oxford\thanksmark{m1}}

\address{Department of Statistics\\
University of Oxford\\
24-29 St$.$ Giles'\\
Oxford OX1 3LB\\
United Kingdom }
\end{aug}

\begin{abstract}
We build on recent works on Stein's method for functions of multivariate normal random variables to derive bounds for the rate of convergence of some asymptotically chi-square distributed statistics.  We obtain some general bounds and establish some simple sufficient conditions for convergence rates of order $n^{-1}$ for smooth test functions.  These general bounds are applied to Friedman's statistic for comparing $r$ treatments across $n$ trials and the family of power divergence statistics for goodness-of-fit across $n$ trials and $r$ classifications, with index parameter $\lambda\in\mathbb{R}$ (Pearson's statistic corresponds to $\lambda=1$).  We obtain a $O(n^{-1})$ bound for the rate of convergence of Friedman's statistic for any number of treatments $r\geq2$.  We also obtain a $O(n^{-1})$ bound on the rate of convergence of the power divergence statistics for any $r\geq2$ when $\lambda$ is a positive integer or any real number greater than 5.  We conjecture that the $O(n^{-1})$ rate holds for any $\lambda\in\mathbb{R}$. 
\end{abstract}



\end{frontmatter}

\section{Introduction}

In this paper, we use Stein's method, introduced in 1972 by Stein \cite{stein}, to obtain bounds on the rate of convergence of some asymptotically chi-square distributed statistics.  In particular, we make use of a recent variant of Stein's method, due to \cite{gaunt normal}, that allows one to obtain approximation theorems when the limit distribution can be represented as a function of multivariate normal random variables.  In this paper, we achieve two goals.  Firstly, we obtain bounds on the rate of convergence of Friedman's statistic and the power divergence family of statistics that improve on those from the existing literature.  Secondly, in deriving these bounds we generalise some of the theory developed in the recent work of \cite{gaunt normal}.  We demonstrate that the theory can be applied in situations in which there is a dependence amongst the random variables of interest (these being, for example, the rankings of treatments across the trials for Friedman's statistic).  We also obtain some simple sufficient conditions for $O(n^{-1})$ convergent rates that weaken those of \cite{gaunt normal}.  The theory developed in this paper allows for the distributional approximation of a large class of statistics (which include the Friedman and Pearson statistics as popular special cases) to be treated within one framework.

\subsection{Chi-square statistics for complete block designs}

In this paper, we study the rate of convergence of a class of statistics for non-parametric tests for complete block designs.  That is, statistics for comparing $r$ treatments or classifications across $n$ independent trials.  In Section 2, we develop some general theory and in Section 3 this theory is applied to several common chi-square statistics, which we now present.

\subsubsection{Friedman's chi-square statistic}

Friedman's chi-square test \cite{friedman} is a non-parametric statistical test that, given $r$ treatments across $n$ independent trials,  can be used to test the null hypothesis that there is no treatment effect against the general alternative.  Suppose that for the $i$-th trial we have the ranking $\pi_i(1),\ldots,\pi_i(r)$, where $\pi_i(j)\in\{1,\ldots,r\}$, over the $r$ treatments.  Under the null hypothesis, the rankings are independent permutations $\pi_1,\ldots,\pi_n$, with each permutation being equally likely.  Let $X_{ij}=\frac{\sqrt{12}}{\sqrt{r(r+1)}}\big(\pi_i(j)-\frac{r+1}{2}\big)$ and set $W_j=\frac{1}{\sqrt{n}}\sum_{i=1}^n X_{ij}$.  Then the Friedman chi-square statistic, given by
\begin{equation}\label{miltonf}F_r=\sum_{j=1}^rW_j^2,
\end{equation}
is asymptotically $\chi_{(r-1)}^2$ distributed under the null hypothesis.

\subsubsection{Pearson's chi-square and the power divergence family of statistics}

Another non-parametric test for complete block designs is Pearson's chi-square goodness-of-fit test, introduced in \cite{pearson}.  Consider $n$ independent trials, with each trial leading to a unique classification over $r$ classes.  Let  $p_1,\ldots,p_r$ represent the non-zero classification probabilities, and let $U_1,\ldots,U_r$ represent the observed numbers arising in each class.  Then Pearson's chi-square statistic, given by
\begin{equation} \chi^2 = \sum_{j=1}^r \frac{(U_j - n p_j)^2}{n p_j}, \end{equation}
is asymptotically $\chi_{(r-1)}^2$ distributed.  Pearson's statistic is a special case ($\lambda=1$) of the so-called power divergence family of statistics introduced by \cite{cr84}:
\begin{equation}\label{powerdiv}T_\lambda(\mathbf{W})=\frac{2}{\lambda(\lambda+1)}\sum_{j=1}^rU_j\bigg[\bigg(\frac{U_j}{np_j}\bigg)^\lambda-1\bigg].
\end{equation}
The statistic $T_\lambda(\mathbf{W})$ is asymptotically $\chi_{(r-1)}^2$ distributed for all $\lambda\in\mathbb{R}$.  When $\lambda=0,-1$, the notation (\ref{powerdiv}) should be understood a result of passage to the limit (see \cite{ulyanov}, Remark 1).  Indeed, the case $\lambda=0$ corresponds to the log-likelihood ratio statistic and the case $\lambda=-1/2$ is the Freeman-Tukey statistic (see \cite{ulyanov}, Remark 2).

\subsubsection{Rates of convergence of chi-square statistics}

In the existing literature, the best bound on the rate of convergence of Friedman's statistic is the following Kolmogorov distance bound of \cite{jensen}:
\begin{equation*}\sup_{z\geq0}|\mathbb{P}(F_r\leq z)-\mathbb{P}(Y\leq z)|\leq C(r)n^{-r/(r+1)},
\end{equation*}
where the (non-explicit) constant $C(r)$ depends only on $r$ and $Y\sim\chi_{(r-1)}^2$. The rate of convergence of other asymptotically chi-square distributed has also received attention in the literature.  For Pearson's statistic over $n$ independent trials with $r$ classifications, it was shown by \cite{yarnold} using Edgeworth expansions that the rate of convergence of Pearson's statistic, in the Kolmogorov distance was $O(n^{(r-1)/r})$, for $r\geq2$, which was improved by \cite{gu03} to $O(n^{-1})$ for $r\geq6$.  Also,  \cite{ulyanov} and \cite{asylbekov} have used Edgeworth expansions to study the rate of convergence of the more general power divergence family of statistics.  For $r\geq4$,  \cite{ulyanov} obtained a $O(n^{(r-1)/r})$ bound on the rate of convergence in the Kolmogorov distance and, for $r=3$, \cite{asylbekov} obtained a $O(n^{-3/4+0.065})$ bound on the rate of convergence in the same metric, with both bounds holding for all $\lambda\in\mathbb{R}$.

To date, the application of Stein's method to the problem of determining rates of convergence of asymptotically chi-square distributed statistics has been quite limited.  In particular, there has been no application to Friedman's statistic in the literature.  Pearson's statistic, however, has received some treatment.    An investigation is given in the unpublished papers \cite{mann} and \cite{mann2}, with a $O(n^{-1/2})$  Kolmogorov distance bound given for Pearson's statistic with general null distribution.  In a recent work \cite{gaunt chi square}, a bound of order $n^{-1}$, for smooth test functions, was obtained.  This bound is valid under any null distribution provided $r\geq2$, and involves the classification probabilities, under the null model, $p_1,\ldots,p_r$ correctly in the sense that the bound goes to zero if and only if $np_*\rightarrow\infty$, where $p_*=\min_{1\leq i\leq r}p_i$.  

In this paper, we obtain bounds for the rate of convergence of a class of statistics for complete block designs that includes the Friedman and Pearson statistics and the power divergence family of statistics, at least for certain values of the index parameter $\lambda$.    By building on the proof techniques of \cite{gaunt normal} and \cite{gaunt chi square}, we establish some simple conditions under which the rate of convergence is of order $n^{-1}$ for smooth test functions, and present the general $O(n^{-1})$ bounds in Theorems \ref{thmsec2}, \ref{thmsec21} and \ref{thmsec3}.  In Section 3, we consider the application of these general bounds to particular chi-square statistics.  In particular, in Theorem \ref{thm1}, we obtain an explicit $O(n^{-1})$ bound on the distributional distance between Friedman's statistic and its limiting chi-square distribution, for smooth test functions.  In Theorem \ref{thmpd}, we also show that the rate of convergence of the power divergence family of statistics is $O(n^{-1})$ for the cases that the index parameter $\lambda$ is either a positive integer or any real number greater than 5.  It is conjectured that this rate holds for any $\lambda\in\mathbb{R}$.

\subsection{Elements of Stein's method for functions of multivariate normal random variables}

To derive our approximation theorems for Friedman's statistic, we employ the powerful probabilistic technique Stein's method.  Originally developed for normal approximation by \cite{stein}, the method has since been extended to many other distributions, such as the multinomial \cite{loh}, exponential \cite{chatterjee, pekoz1}, gamma \cite{gaunt chi square, luk, nourdin1}, variance-gamma \cite{gaunt} and multivariate normal \cite{barbour2, gotze}.  For a comprehensive overview of the current literature and an outline of the basic method see \cite{ley}.  We now outline how Stein's method can be used to prove approximation theorems when the limit distribution can be represented as a function of multivariate normal random variables (for more details see \cite{gaunt normal}). We describe the general approach and explain how it can be applied to statistics and for block designs, such as Friedman's statistic.

Let $g:\mathbb{R}^d\rightarrow\mathbb{R}$ be continuous and let $\mathbf{Z}$ denote the standard $d$-dimensional multivariate normal distribution.  Let $\Sigma$ be non-negative definite, and $\Sigma^{1/2}$ be the unique non-negative matrix so that $\Sigma^{1/2}\mathbf{Z}\sim \mathrm{MVN}(\mathbf{0},\Sigma)$.  Suppose that we are interested in bounding the distributional distance between $g(\mathbf{W})$ and $g(\Sigma^{1/2}\mathbf{Z})$, where $\mathbf{W}\stackrel{D}{\rightarrow}\Sigma^{1/2}\mathbf{Z}$.  To see, for example, that Friedman's statistic falls into this framework, note that $F_r$ can be written in the form $g(\mathbf{W})$, where $g(\mathbf{w})=\sum_{j=1}^rw_j^2$ and the $W_i$ are asymptotically normally distributed by the central limit theorem.  Now, consider the multivariate normal Stein equation (see \cite{goldstein1}) with test function $h(g(\cdot))$:
\begin{equation} \label{mvng} \nabla^T\Sigma\nabla f(\mathbf{w})-\mathbf{w}^T\nabla f(\mathbf{w})=h(g(\mathbf{w}))-\mathbb{E}h(g(\Sigma^{1/2}\mathbf{Z})).
\end{equation} 
We can therefore bound the quantity of interest $|\mathbb{E}h(g(\mathbf{W}))-\mathbb{E}h(g(\Sigma^{1/2}\mathbf{Z}))|$ by solving (\ref{mvng}) for $f$ and then bounding the expectation 
\begin{equation}\label{emvn}\mathbb{E}[\nabla^T\Sigma\nabla f(\mathbf{W})-\mathbf{W}^T\nabla f(\mathbf{W})].
\end{equation}
A number of coupling techniques have been developed for bounding such expectations (see \cite{chatterjee 3, goldstein 2, goldstein1, meckes, reinert 1}).  These papers also give general plug-in bounds for this quantity, although these only hold for the classical case that the derivatives of the test function (here $h(g(\cdot))$) are bounded, in which standard bounds for the derivatives of the solution to (\ref{mvng}) can be applied (see \cite{gaunt rate, goldstein1, meckes}).   However, in general the derivatives of the test function $h(g(\cdot))$ will be unbounded (this is the case for Friedman's statistic) and therefore the derivatives of the solution 
\begin{equation}\label{mvnsolnh}f(\mathbf{w})=-\int_{0}^{\infty}[\mathbb{E}h(g(\mathrm{e}^{-s}\mathbf{w}+\sqrt{1-\mathrm{e}^{-2s}}\Sigma^{1/2}\mathbf{Z}))-\mathbb{E}h(g(\Sigma^{1/2}\mathbf{Z}))]\,\mathrm{d}s
\end{equation}
will also in general be unbounded.  The partial derivatives of the solution (\ref{mvnsolnh}) were bounded by \cite{gaunt normal} for a large class of function $g:\mathbb{R}^d\rightarrow\mathbb{R}$; in particular, bounds are given for the case that the partial derivatives of $g$ have polynomial growth.  These bounds are relevant to our study and are stated in Lemma \ref{cor28}.  With such bounds on the solution and the coupling strategies developed for multivariate normal approximation it is in principle possible to bound the expectation (\ref{emvn}), although we cannot directly apply the existing plug-in bounds.  This is the approach we shall take when obtaining our general approximation theorems in Section 2.

\subsection{Outline of the paper}

In Section 2, we derive general bounds for the distributional distance between statistics $g(\mathbf{W})$ for complete block designs and their limiting distribution $g(\Sigma^{1/2}\mathbf{Z})$.  We give two general $O(n^{-1/2})$ bounds, one for the case of non-negative covariance matrices (Theorem \ref{theoremsec21}) and another for positive definite covariance matrices (Theorem \ref{theoremsec22}).  When the function $g$ is even ($g(\mathbf{w})=g(-\mathbf{w})$ for all $w\in\mathbb{R}^d$), the rate of convergence can be improved to $O(n^{-1})$ for smooth test functions (see Theorem \ref{thmsec2}).  In Section 2.3, we see that it is possible to obtain $O(n^{-1})$ bounds when the assumption that $g$ is relaxed a little (see Theorem \ref{thmsec3}).  In Section 3, we consider the application of the general bounds of Section 2 to the Friedman and Pearson statistics, as well as the power divergence statistics.  In particular, in Theorem \ref{thm1}, we obtain an explicit $O(n^{-1})$ bound for the distributional distance between Friedman's statistics and its limiting chi-square distribution.  In Theorem \ref{thmpd}, we obtain a $O(n^{-1})$ bound on the rate of convergence for the family of power divergence statistics for the cases that $\lambda$ is a positive integer or any real number greater than 5.  We end by conjecturing that this rate holds for all $\lambda\in\mathbb{R}$.

\section{General bounds for the distributional distance between $g(\mathbf{W})$ and $g(\Sigma^{1/2}\mathbf{Z})$}

\subsection{Preliminary lemmas}

Let $X_{ij}$, $i=1,\ldots,n$, $j=1,\ldots,d$, be random variables which have mean zero, but which are not necessarily independent or identically distributed.  Indeed, we shall suppose that $X_{1,j},\ldots,X_{n,j}$ are independent for a fixed $j$, but that the random variables $X_{i,1},\ldots,X_{i,d}$ may be dependent for any fixed $i$.  For $j=1,\ldots,d$, let $W_j=\frac{1}{\sqrt{n}}\sum_{i=1}^nX_{ij}$ and denote $\mathbf{W}=(W_1,\ldots,W_d)^T$.  To deal with this dependence structure, we introduce the random variables $W_j^{(i)}=W_j-\frac{1}{\sqrt{n}}X_{ij}$, so that $W_j^{(i)}$ and $X_{ij}$ are independent.  We also write $\mathbf{W}^{(i)}=(W_1^{(i)},\ldots,W_d^{(i)})^T$.  Suppose that the covariance matrix $\Sigma$ of $\mathbf{W}$ is non-negative definite.  Let $\mathbf{Z}$ have the standard $d$-dimensional multivariate normal distribution, so that $\Sigma^{1/2}\mathbf{Z}\sim\mathrm{MVN}(\mathbf{0},\Sigma)$.  Let $\sigma_{jk}=(\Sigma)_{jk}$ and $Z_i=(\Sigma^{1/2}\mathbf{Z})_i\sim N(0,\sigma_{ii})$.  

In this section, we shall obtain bounds on the distributional distance between $g(\mathbf{W})$ and $g(\Sigma^{1/2}\mathbf{Z})$, where $g:\mathbb{R}^d\rightarrow\mathbb{R}$ is a sufficiently differentiable function.  Note that $g(\mathbf{W})$ takes the form of a statistic for complete block designs.   In this subsection, we give two bounds: one for general $g$ and a second for the case that $g$ is an even function ($g(\mathbf{w})=g(\mathbf{-w})$ for all $\mathbf{w}\in\mathbb{R}^d$).  In the next subsection, we shall specialise to the case that the partial derivatives of $g$ have polynomial growth.  Before presenting our bounds, we introduce some notation.  We shall let $C^k(I)$ denote the class of real-valued functions defined on $I\subseteq\mathbb{R}^d$ whose partial derivatives of order $k$ all exist.  We shall also let $C_b^k(I)$ denote the class of real-valued functions defined on $I\subseteq\mathbb{R}^d$ whose partial derivatives of order $k$ all exist and are bounded. 

\begin{lemma}\label{noteveng}Let $X_{ij}$, $i=1,\ldots,n$, $j=1,\ldots,d$, be defined as above.  Suppose $h$ and $g$ are such that $f\in C^{3}(\mathbb{R}^d)$, where $f$ is given by (\ref{mvnsolnh}).  Then, if the expectations on the right-hand side of (\ref{springz}) exist, 
\begin{align}&|\mathbb{E}h(g(\mathbf{W}))-\mathbb{E}h(g(\Sigma^{1/2}\mathbf{Z}))|\nonumber\\
&\leq\frac{1}{2n^{3/2}}\sum_{i=1}^n\sum_{j,k,l=1}^d\bigg\{\sup_{\theta}\mathbb{E}\bigg|X_{ij}X_{ik}X_{il}\frac{\partial^3f}{\partial w_j\partial w_k\partial w_l}(\mathbf{W}_{\theta}^{(i)})\bigg|\nonumber\\
\label{springz}&\quad+2|\mathbb{E}X_{ij}X_{ik}|\sup_{\theta}\mathbb{E}\bigg|X_{il}\frac{\partial^3f}{\partial w_j\partial w_k\partial w_l}(\mathbf{W}_{\theta}^{(i)})\bigg|\bigg\},
\end{align}
where $\mathbf{W}_{\theta}^{(i)}=\mathbf{W}^{(i)}+\frac{\theta}{\sqrt{n}}\mathbf{X}_{i}$ for some $\theta\in(0,1)$ and $\mathbf{X}_{i}=(X_{i,1},\ldots,X_{i,d})^T$.
\end{lemma}

\begin{proof}We aim to bound $\mathbb{E}h(g(\mathbf{W}))-\mathbb{E}h(g(\mathbf{Z}))$, and do so by bounding the quantity
\begin{equation*}\mathbb{E}[\nabla^T\Sigma\nabla f(\mathbf{W})-\mathbf{W}^T\nabla f(\mathbf{W})]=\mathbb{E}\bigg[\sum_{j,k=1}^d\sigma_{jk}\frac{\partial^2f}{\partial w_j\partial w_k}(\mathbf{W})-\sum_{j=1}^dW_j\frac{\partial f}{\partial w_j}(\mathbf{W})\bigg].
\end{equation*}
Taylor expanding $\frac{\partial f}{\partial w_j}(\mathbf{W})$ about $\mathbf{W}_j^{(i)}$ gives
\begin{align*}&\sum_{j=1}^d\mathbb{E}W_j\frac{\partial f}{\partial w_j}(\mathbf{W})=\frac{1}{\sqrt{n}}\sum_{i=1}^n\sum_{j=1}^d\mathbb{E}X_{ij}\frac{\partial f}{\partial w_j}(\mathbf{W}) \\
&=\frac{1}{\sqrt{n}}\sum_{i=1}^n\sum_{j=1}^d\mathbb{E}X_{ij}\frac{\partial f}{\partial w_j}(\mathbf{W}^{(i)})+\frac{1}{n}\sum_{i=1}^n\sum_{j,k=1}^d\mathbb{E}X_{ij}X_{ik}\frac{\partial^2f}{\partial w_j\partial w_k}(\mathbf{W}^{(i)})+R_1 \\
&=\frac{1}{\sqrt{n}}\sum_{i=1}^n\sum_{j=1}^d\mathbb{E}X_{ij}\mathbb{E}\frac{\partial f}{\partial w_j}(\mathbf{W}^{(i)})+\frac{1}{n}\sum_{i=1}^n\sum_{j,k=1}^d\mathbb{E}X_{ij}X_{ik}\mathbb{E}\frac{\partial^2f}{\partial w_j\partial w_k}(\mathbf{W}^{(i)})+R_1 \\
&=\frac{1}{n}\sum_{i=1}^n\sum_{j,k=1}^d\mathbb{E}X_{ij}X_{ik}\mathbb{E}\frac{\partial^2f}{\partial w_j\partial w_k}(\mathbf{W})+R_1+R_2\\
&=\sum_{j,k=1}^d\sigma_{jk}\mathbb{E}\frac{\partial^2f}{\partial w_j\partial w_k}(\mathbf{W})+R_1+R_2,
\end{align*}
where
\begin{align*}|R_1|&\leq\frac{1}{2n^{3/2}}\sum_{i=1}^n\sum_{j,k,l=1}^d\sup_{\theta}\mathbb{E}\bigg|X_{ij}X_{ik}X_{il}\frac{\partial^3f}{\partial w_j\partial w_k\partial w_l}(\mathbf{W}_{\theta}^{(i)})\bigg|, \\
|R_2|&\leq\frac{1}{n^{3/2}}\sum_{i=1}^n\sum_{j,k,l=1}^d|\mathbb{E}X_{ij}X_{ik}|\sup_{\theta}\mathbb{E}\bigg|X_{il}\frac{\partial^3f}{\partial w_j\partial w_k\partial w_l}(\mathbf{W}_{\theta}^{(i)})\bigg|.
\end{align*}
Here we used that $\frac{1}{n}\sum_{i=1}^n\mathbb{E}X_{ij}X_{ik}=\mathbb{E}W_jW_k=\sigma_{jk}$.  The proof is complete.
\end{proof}

\begin{remark}In the statement of Lemma \ref{noteveng}, we did not give precise conditions on $h$ and $g$ such that $f\in C_b^{3}(\mathbb{R}^d)$, nor restrictions on the $X_{ij}$ such that the expectations on the right-hand side of (\ref{springz}) exist.  In applying, Lemma \ref{noteveng} in practice (see Section 2.2), one would need to check that $h$, $g$ and the $X_{ij}$ are such that these conditions are met.  The same comment applies equally to Lemma \ref{evennormal}. 
\end{remark}

We now obtain an analogue of Lemma \ref{noteveng} for the case that $g$ is an even function.  The symmetry of the function $g$ allows us to obtain $O(n^{-1})$ convergence rates for smooth test functions $h$.  The following partial differential equation 
\begin{equation}\label{234multinor}\nabla^T\Sigma\nabla \psi_{jkl}(\mathbf{w})-\mathbf{w}^T\nabla \psi_{jkl}(\mathbf{w})=\frac{\partial^{3} f}{\partial w_j\partial w_k\partial w_l}(\mathbf{w})
\end{equation}
shall appear in our proof.

\begin{lemma}\label{evennormal}Let $X_{ij}$, $i=1,\ldots,n$, $j=1,\ldots,d$, be defined as they were for Lemma \ref{noteveng}. Suppose $g:\mathbb{R}^d\rightarrow\mathbb{R}$ is an even function.  Suppose further that the solution (\ref{mvnsolnh}), denoted by $f$, belongs to the class $C^{4}(\mathbb{R}^d)$ and that the solution $\psi_{jkl}$ to (\ref{234multinor}) is in the class $C^3(\mathbb{R}^d)$.  Then, if the expectations on the right-hand side of (\ref{dig hole}) exist,
\begin{align}&|\mathbb{E}h(g(\mathbf{W}))-\mathbb{E}h(g(\Sigma^{1/2}\mathbf{Z}))|\nonumber \\
&\leq\frac{1}{6n^2}\sum_{i=1}^n\sum_{j,k,l,t=1}^d\bigg\{\sup_{\theta}\mathbb{E}\bigg|X_{ij}X_{ik}X_{il}X_{it}\frac{\partial^4f}{\partial w_j\partial w_k\partial w_l\partial w_t}(\mathbf{W}_{\theta}^{(i)})\bigg|\nonumber\\
&+9|\mathbb{E}X_{ij}X_{ik}|\sup_{\theta}\mathbb{E}\bigg|X_{il}X_{it}\frac{\partial^4f}{\partial w_j\partial w_k\partial w_l\partial w_t}(\mathbf{W}_{\theta}^{(i)})\bigg|\nonumber \\
&+3|\mathbb{E}X_{ij}X_{ik}X_{il}|\sup_{\theta}\mathbb{E}\bigg|X_{it}\frac{\partial^4f}{\partial w_j\partial w_k\partial w_l\partial w_t}(\mathbf{W}_{\theta}^{(i)})\bigg|\bigg\}\nonumber \\
&+\frac{1}{4n^3}\sum_{i=1}^n\sum_{j,k,l=1}^d|\mathbb{E}X_{ij}X_{ik}X_{il}|\sum_{\alpha=1}^n\sum_{a,b,c=1}^d\bigg\{\sup_{\theta}\mathbb{E}\bigg|X_{\alpha a}X_{\alpha b}X_{\alpha c}\frac{\partial^3\psi_{jkl}}{\partial w_a\partial w_b\partial w_c}(\mathbf{W}_{\theta}^{(i)})\bigg| \nonumber\\
\label{dig hole}&+2|\mathbb{E}X_{\alpha a}X_{\alpha b}|\sup_{\theta}\mathbb{E}\bigg|X_{\alpha c}\frac{\partial^3\psi_{jkl}}{\partial w_a\partial w_b\partial w_c}(\mathbf{W}_{\theta}^{(i)})\bigg|\bigg\}.
\end{align}
\end{lemma}

\begin{proof}By a similar argument to the one used in the proof of Lemma \ref{noteveng}, 
\begin{align*}\sum_{j=1}^d\mathbb{E}W_j\frac{\partial f}{\partial w_j}(\mathbf{W})&=
\frac{1}{n}\sum_{i=1}^n\sum_{j,k=1}^d\mathbb{E}X_{ij}X_{ik}\mathbb{E}\frac{\partial^2f}{\partial w_j\partial w_k}(\mathbf{W}^{(i)})+N_1+R_1 \\
&=\frac{1}{n}\sum_{i=1}^n\sum_{j,k=1}^d\mathbb{E}X_{ij}X_{ik}\mathbb{E}\frac{\partial^2f}{\partial w_j\partial w_k}(\mathbf{W})+N_1+N_2+R_1 +R_2,
\end{align*}
where
\begin{align*}N_1&=\frac{1}{2n^{3/2}}\sum_{i=1}^n\sum_{j,k,l=1}^d\mathbb{E}X_{ij}X_{ik}X_{il}\mathbb{E}\frac{\partial^3f}{\partial w_j\partial w_k\partial w_l}(\mathbf{W}^{(i)}),\\
N_2&=-\frac{1}{n^{3/2}}\sum_{i=1}^n\sum_{j,k,l=1}^d\mathbb{E}X_{ij}X_{ik}\mathbb{E}X_{il}\frac{\partial^3f}{\partial w_j\partial w_k\partial w_l}(\mathbf{W}),\\
|R_1|&\leq \frac{1}{6n^2}\sum_{i=1}^n\sum_{j,k,l,t=1}^d\sup_{\theta}\mathbb{E}\bigg|X_{ij}X_{ik}X_{il}X_{it}\frac{\partial^4f}{\partial w_j\partial w_k\partial w_l\partial w_t}(\mathbf{W}_{\theta}^{(i)})\bigg|, \\
|R_2|&\leq \frac{1}{2n^2}\sum_{i=1}^n\sum_{j,k,l,t=1}^d|\mathbb{E}X_{ij}X_{ik}|\sup_{\theta}\mathbb{E}\bigg|X_{il}X_{it}\frac{\partial^4f}{\partial w_j\partial w_k\partial w_l\partial w_t}(\mathbf{W}_{\theta}^{(i)})\bigg|.
\end{align*}
We can write $N_1$ as
\begin{equation*}N_1=\frac{1}{2n^{3/2}}\sum_{i=1}^n\sum_{j,k,l=1}^d\mathbb{E}X_{ij}X_{ik}X_{il}\mathbb{E}\frac{\partial^3f}{\partial w_j\partial w_k\partial w_l}(\mathbf{W})+R_3,
\end{equation*}
where
\begin{equation*}|R_3|\leq \frac{1}{2n^2}\sum_{i=1}^n\sum_{j,k,l,t=1}^d|\mathbb{E}X_{ij}X_{ik}X_{il}|\sup_{\theta}\mathbb{E}\bigg|X_{it}\frac{\partial^4f}{\partial w_j\partial w_k\partial w_l\partial w_t}(\mathbf{W}_{\theta}^{(i)})\bigg|,
\end{equation*}
and we can also write $N_2$ as
\begin{equation*}N_2=-\frac{1}{n^{3/2}}\sum_{i=1}^n\sum_{j,k,l=1}^d\mathbb{E}X_{ij}X_{ik}\mathbb{E}X_{il}\mathbb{E}\frac{\partial^3f}{\partial w_j\partial w_k\partial w_l}(\mathbf{W}^{(i)})+R_4=R_4,
\end{equation*}
where
\begin{equation*}|R_4|\leq\frac{1}{n^2}\sum_{i=1}^n\sum_{j,k,l,t=1}^d|\mathbb{E}X_{ij}X_{ik}|\sup_{\theta}\mathbb{E}\bigg|X_{il}X_{it}\frac{\partial^4f}{\partial w_j\partial w_k\partial w_l\partial w_t}(\mathbf{W}_{\theta}^{(i)})\bigg|.
\end{equation*}
Combining bounds gives that
\begin{align}&|\mathbb{E}h(g(\mathbf{W}))-\mathbb{E}h(g(\Sigma^{1/2}\mathbf{Z}))|\nonumber\\
&\leq \frac{1}{2n^{3/2}}\sum_{i=1}^n\sum_{j,k,l=1}^d|\mathbb{E}X_{ij}X_{ik}X_{il}|\bigg|\mathbb{E}\frac{\partial^3f}{\partial w_j\partial w_k\partial w_l}(\mathbf{W})\bigg|\nonumber\\
\label{nearlyeven}&\quad+|R_1|+|R_2|+|R_3|+|R_4|.
\end{align}
To achieve the desired $O(n^{-1})$ bound we need to show that $\mathbb{E}\frac{\partial^{3} f}{\partial w_j\partial w_k\partial w_l}(\mathbf{W})$ is of order $n^{-1/2}$, since in general $\mathbb{E}X_{ij}X_{ik}X_{il}\not=0$.  We consider the $\mathrm{MVN}(\mathbf{0},\Sigma)$ Stein equation with test function $\frac{\partial^{3} f}{\partial w_j\partial w_k\partial w_l}$:
\begin{equation*}\nabla^T\Sigma\nabla \psi_{jkl}(\mathbf{w})-\mathbf{w}^T\nabla \psi_{jkl}(\mathbf{w})=\frac{\partial^{3} f}{\partial w_j\partial w_k\partial w_l}(\mathbf{w})-\mathbb{E}\bigg[\frac{\partial^{3} f}{\partial w_j\partial w_k\partial w_l}(\Sigma^{1/2}\mathbf{Z})\bigg].
\end{equation*}
Since $g$ is an even function, the solution $f$, as given by (\ref{mvnsolnh}), is an even function (see \cite{gaunt normal}, Lemma 3.2).  Therefore $\mathbb{E}\frac{\partial^{3} f}{\partial w_j\partial w_k\partial w_l}(\Sigma^{1/2}\mathbf{Z})=0$, and so
\begin{equation}\label{mpat}\mathbb{E}\bigg[\frac{\partial^{3} f}{\partial w_j\partial w_k\partial w_l}(\mathbf{W})\bigg]=\mathbb{E}[\nabla^T\Sigma\nabla \psi_{jkl}(\mathbf{W})-\mathbf{W}^T\nabla \psi_{jkl}(\mathbf{W})].
\end{equation}
We can use Lemma \ref{noteveng} to bound the right-hand side of (\ref{mpat}), which allows us to obtain a $O(n^{-1/2})$ bound for this quantity.  All terms have now been bounded to the desired order and the proof is complete.
\end{proof}

\subsection{Approximation theorems for polynomial $P$}

Lemmas \ref{noteveng} and \ref{evennormal} allow one to bound the distributional distance between $g(\mathbf{W})$ and $g(\Sigma^{1/2}\mathbf{Z})$ if bounds are available for the expectations on the right-hand side of (\ref{springz}) and (\ref{dig hole}), respectively.  In this subsection, we obtain such bounds for the case that the partial derivatives of $g$ have polynomial growth.  We begin, with Lemma \ref{cor28} (below), in which we state some bounds (see \cite{gaunt normal}, Corollary 2.2 and 2.3) for the solutions $f$ and $\psi_{jkl}$.  In \cite{gaunt normal} bounds for $f$ and $\psi_{jkl}$ are also available for the case that the partial derivatives of $g$ have exponential growth, although for space reasons we do not include these bounds (polynomial bounds suffice for our applications).  

We say that the function $g:\mathbb{R}^d\rightarrow\mathbb{R}$ belongs to the class $C_{P}^m(\mathbb{R}^d)$ if all $m$-th order partial derivatives of $g$ exist and there exists a dominating function $P:\mathbb{R}^d\rightarrow\mathbb{R}^+$ such that, for all $\mathbf{w}\in\mathbb{R}^d$, the partial derivatives satisfy
\begin{equation*}\bigg|\frac{\partial^kg(\mathbf{w})}{\prod_{j=1}^k\partial w_{i_j}}\bigg|^{m/k}\leq P(\mathbf{w}):=A+\sum_{i=1}^dB_i|w_i|^{r_i}, \quad k=1,\ldots,m,
\end{equation*}
where $A\geq 0$, $B_1,\ldots,B_d\geq 0$ and $r_1,\ldots,r_d\geq 0$.  We shall write $h_m=\sum_{j=1}^m{m\brace j}\|h^{(j)}\|$, where $\|h\|=\|h\|_\infty=\sup_{w\in\mathbb{R}}|h(w)|$ and the Stirling numbers of the second kind are given by ${m\brace j}=\frac{1}{j!}\sum_{i=0}^j(-1)^{j-i}\binom{j}{i}i^m$ (see \cite{olver}).

\begin{lemma}\label{cor28}Suppose $\Sigma$ is positive definite and $h\in C_b^{m-1}(\mathbb{R})$ and $g\in C_P^{m-1}(\mathbb{R}^d)$ for $m\geq 2$.  Let $Z_i=(\Sigma^{1/2}\mathbf{Z})_i\sim N(0,\sigma_{ii})$.  Then, for all $\mathbf{w}\in\mathbb{R}^d$,
\begin{align}
\bigg|\frac{\partial^{m}f(\mathbf{w})}{\prod_{j=1}^{m}\partial w_{i_j}}\bigg|&\leq h_{m-1}\min_{1\leq l\leq d}
\bigg[A\mathbb{E}|(\Sigma^{-1/2}\mathbf{Z})_l|+\sum_{i=1}^d2^{r_i}B_i\big(|w_i|^{r_i}\mathbb{E}|(\Sigma^{-1/2}\mathbf{Z})_l|\nonumber\\
\label{sigmaiii1}&\quad+\mathbb{E}|(\Sigma^{-1/2}\mathbf{Z})_lZ_i^{r_i}|\big)\bigg].
\end{align}
Suppose now that $\Sigma$ is non-negative definite and $h\in C_b^m(\mathbb{R})$ and $g\in C_P^m(\mathbb{R}^d)$ for $m\geq 1$.  Then, for all $\mathbf{w}\in\mathbb{R}^d$,
\begin{align}\label{sigmaiii}\bigg|\frac{\partial^mf(\mathbf{w})}{\prod_{j=1}^m\partial w_{i_j}}\bigg|&\leq\frac{h_m}{m}\bigg[A+\sum_{i=1}^d2^{r_i}B_i\big(|w_i|^{r_i}+\mathbb{E}|Z_i|^{r_i}\big)\bigg].
\end{align}
Let $h\in C_b^{6}(\mathbb{R})$ and $g\in C_P^{6}(\mathbb{R}^d)$.  Then, for all $\mathbf{w}\in\mathbb{R}^d$,
\begin{align}\label{sigmaiii2}\bigg|\frac{\partial^3\psi_{jkl}(\mathbf{w})}{\partial w_a\partial w_b\partial w_c}\bigg|&\leq\frac{h_{6}}{18}\bigg[A+\sum_{i=1}^d3^{r_i}B_i\big(|w_i|^{r_i}+2\mathbb{E}|Z_i|^{r_i}\big)\bigg].
\end{align}
\end{lemma} 

\begin{lemma}\label{cbwbshc} Let $\mathbf{X}_i$ denote the vector $(X_{i,1},\ldots,X_{i,d})^T$ and let $u:\mathbb{R}^d\rightarrow\mathbb{R}^+$ be such that $\mathbb{E}|X_{ij}^{r_j}u(\mathbf{X}_i)|<\infty$ for all $i=1,\ldots,n$ and $j=1,\ldots,d$.  Then, for all $\theta\in(0,1)$,
\begin{align}\mathbb{E}\bigg|u(\mathbf{X}_i)\frac{\partial^mf}{\prod_{j=1}^{m}\partial w_{i_j}}(\mathbf{W}_{\theta}^{(i)})\bigg| &\leq \frac{h_m}{m}\bigg[A\mathbb{E}u(\mathbf{X}_i)+\sum_{j=1}^d2^{r_j}B_j\bigg(2^{r_j}\mathbb{E}u(\mathbf{X}_i)\mathbb{E}|W_j|^{r_j} \nonumber \\
\label{cbhsxx}&\quad+\frac{2^{r_j}\mathbb{E}|X_{ij}^{r_j}u(\mathbf{X}_i)|}{n^{r_j/2}}+\mathbb{E}|Z_j|^{r_j}\mathbb{E}u(\mathbf{X}_i)\bigg)\bigg], \\
\mathbb{E}\bigg|u(\mathbf{X}_i)\frac{\partial^mf}{\prod_{j=1}^{m}\partial w_{i_j}}(\mathbf{W}_{\theta}^{(i)})\bigg| &\leq h_{m-1}\min_{1\leq l\leq d}\mathbb{E}|(\Sigma^{-1/2}\mathbf{Z})_l|\bigg[A\mathbb{E}u(\mathbf{X}_i)\nonumber\\
&\quad+\sum_{j=1}^d2^{r_j}B_j\bigg(2^{r_j}\mathbb{E}u(\mathbf{X}_i)\mathbb{E}|W_j|^{r_j} \nonumber \\
&\quad+\frac{2^{r_j}\mathbb{E}|X_{ij}^{r_j}u(\mathbf{X}_i)|}{n^{r_j/2}}+\frac{\mathbb{E}|(\Sigma^{-1/2}\mathbf{Z})_lZ_j^{r_j}|}{\mathbb{E}|(\Sigma^{-1/2}\mathbf{Z})_l|}\mathbb{E}u(\mathbf{X}_i)\bigg)\bigg],\nonumber \\
\mathbb{E}\bigg|u(\mathbf{X}_i)\frac{\partial^3\psi_{abc}}{\partial w_j\partial w_k\partial w_l}(\mathbf{W}_{\theta}^{(i)})\bigg| &\leq \frac{h_{6}}{18}\bigg[A\mathbb{E}u(\mathbf{X}_i)+\sum_{j=1}^d3^{r_j}B_j\bigg(2^{r_j}\mathbb{E}u(\mathbf{X}_i)\mathbb{E}|W_j|^{r_j} \nonumber \\
&\quad+\frac{2^{r_j}\mathbb{E}|X_{ij}^{r_j}u(\mathbf{X}_i)}{n^{r_j/2}}|+2\mathbb{E}|Z_j|^{r_j+1}\mathbb{E}u(\mathbf{X}_i)\bigg)\bigg], \nonumber
\end{align}
where the inequalities are for $g$ in the classes $C_P^k(\mathbb{R}^d)$, $C_P^{k-1}(\mathbb{R}^d)$ and $C_P^{6}(\mathbb{R}^d)$, respectively.  For the second inequality, we must assume that $\Sigma$ is positive definite; for the other inequalities it suffices for $\Sigma$ to be non-negative definite.
\end{lemma}

\begin{proof}Let us prove the first inequality.  From inequality (\ref{sigmaiii}) we have
\begin{align*}\mathbb{E}\bigg|u(\mathbf{X}_i)\frac{\partial^mf}{\prod_{j=1}^{m}\partial w_{i_j}}(\mathbf{W}_{\theta}^{(i)})\bigg| &\leq \frac{h_m}{m}\bigg[A\mathbb{E}u(\mathbf{X}_i)+\sum_{j=1}^d2^{r_j}B_j\Big(\mathbb{E}|u(\mathbf{X}_i)(W_{j,\theta}^{(i)})^{r_j}|\\
&\quad+\mathbb{E}|Z|^{r_j}\mathbb{E}u(\mathbf{X}_i)\Big)\bigg],
\end{align*}
where $W_{j,\theta}^{(i)}$ is the $j$-th component of $\mathbf{W}_{\theta}^{(i)}$.   By using the crude inequality $|a+b|^s\leq 2^s(|a|^s+|b|^s)$, which holds for any $s\geq0$, and independence of $X_{ij}$ and $W_j^{(i)}$, we have 
\begin{align}\mathbb{E}|u(\mathbf{X}_i)(W_{j,\theta}^{(i)})^{r_j}|&\leq 2^{r_j}\mathbb{E}\bigg|u(\mathbf{X}_i)\bigg(|W_j^{(i)}|^{r_j}+\frac{\theta^{r_j}}{n^{r_j/2}}|X_{ij}|^{r_j}\bigg)\bigg| \nonumber\\
\label{foxsell}&\leq 2^{r_j}\bigg(\mathbb{E}u(\mathbf{X}_i)\mathbb{E}|W_j^{(i)}|^{r_j}+\frac{1}{n^{r_j/2}}\mathbb{E}|X_{ij}^{r_j}u(\mathbf{X}_i)|\bigg),
\end{align}
Using that $\mathbb{E}|W_j^{(i)}|^{r_j}\leq \mathbb{E}|W_j|^{r_j}$ leads to the desired inequality.  This can be seen by using Jensen's inequality: 
\begin{align*}\mathbb{E}|W_j|^{r_j}&=\mathbb{E}[\mathbb{E}[|W_j^{(i)}+n^{-1/2}X_{ij}|^{r_j} \: | \: W_j^{(i)}]] \\ 
&\geq\mathbb{E}|\mathbb{E}[W_j^{(i)}+n^{-1/2}X_{ij} \: | \: W_{j}^{(i)}]|^{r_j} 
=\mathbb{E}|W_j^{(i)}|^{r_j}.
\end{align*}
Thus we obtain the first inequality.  The proofs of the other two inequalities are similar; we just use inequalities (\ref{sigmaiii1}) and (\ref{sigmaiii2}) instead of inequality (\ref{sigmaiii}).
\end{proof}

By applying the inequalities of Lemma \ref{cbwbshc} to the bounds of Lemmas \ref{noteveng} and \ref{evennormal}, we can obtain the following four theorems for the distributional distance between $g(\mathbf{W})$ and $g(\Sigma^{1/2}\mathbf{Z})$ when the derivatives of $g$ have polynomial growth.  Theorem \ref{theoremsec21} follows from using inequality (\ref{cbhsxx}) in the bound of Lemma \ref{noteveng}, and the other theorems are proved similarly.  Theorems \ref{thmsec2} and \ref{thmsec21} give some simple sufficient conditions under which a $O(n^{-1})$ bound can be obtained for smooth test functions.  We could obtain analogues of Theorems \ref{thmsec2} and \ref{thmsec21} for the case of a positive definite covariance matrix $\Sigma$ (which would impose weaker conditions on $g$ and $h$) by appealing to results from Section 2 of \cite{gaunt normal}.  However, for space reasons, we do not present them (our applications involve covariance matrices that are only non-negative definite).

\begin{theorem}\label{theoremsec21}Let $X_{ij}$, $i=1,\ldots,n$, $j=1,\ldots,d$, be defined as in Lemma \ref{noteveng}, but with the additional assumption that $\mathbb{E}|X_{ij}|^{r_k+3}<\infty$ for all $i$, $j$ and $1\leq k\leq d$.  Suppose $\Sigma$ is non-negative definite and that $g\in C_P^3(\mathbb{R}^d)$.  Let $Z_i=(\Sigma^{1/2}\mathbf{Z})_i\sim N(0,\sigma_{ii})$.  Then, for $h\in C_b^3(\mathbb{R})$, 
\begin{align*}&|\mathbb{E}h(g(\mathbf{W}))-\mathbb{E}h(g(\Sigma^{1/2}\mathbf{Z}))|\\
&\leq\frac{h_3}{6n^{3/2}}\sum_{i=1}^n\sum_{j,k,l=1}^d\bigg\{A\mathbb{E}|X_{ij}X_{ik}X_{il}|+\sum_{t=1}^d 2^{r_t}B_t\bigg(2^{r_t}\mathbb{E}|X_{ij}X_{ik}X_{il}|\mathbb{E}|W_t|^{r_t}\\
&\quad+\frac{2^{r_t}}{n^{r_t/2}}\mathbb{E}|X_{ij}X_{ik}X_{il}X_{it}^{r_t}|+\mathbb{E}|Z_t|^{r_t}\mathbb{E}|X_{ij}X_{ik}X_{il}|\bigg)+2|\mathbb{E}X_{ij}X_{ik}|\bigg[A\mathbb{E}|X_{il}| \\
&\quad+\sum_{t=1}^d 2^{r_t}B_t\bigg(2^{r_t}\mathbb{E}|X_{il}|\mathbb{E}|W_t|^{r_t}+\frac{2^{r_t}}{n^{r_t/2}}\mathbb{E}|X_{il}X_{it}^{r_t}|+\mathbb{E}|Z_t|^{r_t}\mathbb{E}|X_{il}|\bigg)\bigg]\bigg\}.
\end{align*}
\end{theorem}

\begin{theorem}\label{theoremsec22}Let $X_{ij}$, $i=1,\ldots,n$, $j=1,\ldots,d$, be defined as in Lemma \ref{noteveng}, but with the additional assumption that $\mathbb{E}|X_{ij}|^{r_k+3}<\infty$ for all $i$, $j$ and $1\leq k\leq d$.   Suppose $\Sigma$ is positive definite and that $g\in C_P^2(\mathbb{R}^d)$.  Then, for $h\in C_b^2(\mathbb{R})$, 
\begin{align*}&|\mathbb{E}h(g(\mathbf{W}))-\mathbb{E}h(g(\Sigma^{1/2}\mathbf{Z}))|\leq\frac{h_2}{2n^{3/2}}\min_{1\leq s\leq d}\mathbb{E}|(\Sigma^{-1/2}\mathbf{Z})_s|\\
&\times\sum_{i=1}^n\sum_{j,k,l=1}^d\bigg\{A\mathbb{E}|X_{ij}X_{ik}X_{il}|+\sum_{t=1}^r 2^{r_t}B_t\bigg(2^{r_t}\mathbb{E}|X_{ij}X_{ik}X_{il}|\mathbb{E}|W_t|^{r_t}\\
&+\frac{2^{r_t}}{n^{r_t/2}}\mathbb{E}|X_{ij}X_{ik}X_{il}X_{it}^{r_t}|+\frac{\mathbb{E}|(\Sigma^{-1/2}\mathbf{Z})_sZ_t^{r_t}|}{\mathbb{E}|(\Sigma^{-1/2}\mathbf{Z})_s|}\mathbb{E}|X_{ij}X_{ik}X_{il}|\bigg)\\
&+2|\mathbb{E}X_{ij}X_{ik}|\bigg[A\mathbb{E}|X_{il}| 
+\sum_{t=1}^r 2^{r_t}B_t\bigg(2^{r_t}\mathbb{E}|X_{il}|\mathbb{E}|W_t|^{r_t}+\frac{2^{r_t}}{n^{r_t/2}}\mathbb{E}|X_{il}X_{it}^{r_t}|\\
&+\frac{\mathbb{E}|(\Sigma^{-1/2}\mathbf{Z})_sZ_t^{r_t}|}{\mathbb{E}|(\Sigma^{-1/2}\mathbf{Z})_s|}\mathbb{E}|X_{il}|\bigg)\bigg]\bigg\}.
\end{align*}
\end{theorem}

\begin{theorem}\label{thmsec2}Let $X_{ij}$, $i=1,\ldots,n$, $j=1,\ldots,d$, be defined as in Lemma \ref{noteveng}, but with the additional assumption that $\mathbb{E}|X_{ij}|^{r_k+4}<\infty$ for all $i,$ $j$ and $1\leq k\leq d$.  Suppose $\Sigma$ is non-negative definite and that $g\in C_P^6(\mathbb{R}^d)$ is an even function.  Then, for $h\in C_b^6(\mathbb{R})$,  
\begin{align}&|\mathbb{E}h(g(\mathbf{W}))-\mathbb{E}h(g(\Sigma^{1/2}\mathbf{Z}))|\leq M:=\frac{h_4}{24n^{2}}\sum_{i=1}^n\sum_{j,k,l,m=1}^d\bigg\{A\mathbb{E}|X_{ij}X_{ik}X_{il}X_{im}|\nonumber\\
&+\sum_{t=1}^d 2^{r_t}B_t\bigg(2^{r_t}\mathbb{E}|X_{ij}X_{ik}X_{il}X_{im}|\mathbb{E}|W_t|^{r_t}+\frac{2^{r_t}}{n^{r_t/2}}\mathbb{E}|X_{ij}X_{ik}X_{il}X_{im}X_{it}^{r_t}|\nonumber\\
&+\mathbb{E}|Z_t|^{r_t}\mathbb{E}|X_{ij}X_{ik}X_{il}X_{im}|\bigg)+9|\mathbb{E}X_{ij}X_{ik}|\bigg[A\mathbb{E}|X_{il}X_{im}| \nonumber\\
&+\sum_{t=1}^d 2^{r_t}B_t\bigg(2^{r_t}\mathbb{E}|X_{il}X_{im}|\mathbb{E}|W_t|^{r_t}+\frac{2^{r_t}}{n^{r_t/2}}\mathbb{E}|X_{il}X_{im}X_{it}^{r_t}|+\mathbb{E}|Z_t|^{r_t}\mathbb{E}|X_{il}X_{im}|\bigg)\bigg]\nonumber \\
&+3|\mathbb{E}X_{ij}X_{ik}X_{il}|\bigg[A\mathbb{E}|X_{im}|+\sum_{t=1}^d 2^{r_t}B_t\bigg(2^{r_t}\mathbb{E}|X_{im}|\mathbb{E}|W_t|^{r_t}+\frac{2^{r_t}}{n^{r_t/2}}\mathbb{E}|X_{it}^{r_t}|\nonumber\\
&+\mathbb{E}|Z_t|^{r_t}\mathbb{E}|X_{im}|\bigg)\bigg]\bigg\}+\frac{h_6}{72n^3}\sum_{\alpha=1}^n\sum_{a,b,c=1}^d|\mathbb{E}X_{\alpha a}X_{\alpha b}X_{\alpha c}|\sum_{i=1}^n\sum_{j,k,l=1}^d\bigg\{A\mathbb{E}|X_{ij}X_{ik}X_{il}|\nonumber\\
&+\sum_{t=1}^d 3^{r_t}B_t\bigg(2^{r_t}\mathbb{E}|X_{ij}X_{ik}X_{il}|\mathbb{E}|W_t|^{r_t}+\frac{2^{r_t}}{n^{r_t/2}}\mathbb{E}|X_{ij}X_{ik}X_{il}X_{it}^{r_t}| \nonumber\\
&+2\mathbb{E}|Z_t|^{r_t+1}\mathbb{E}|X_{ij}X_{ik}X_{il}|\bigg)+2|\mathbb{E}X_{ij}X_{ik}|\bigg[A\mathbb{E}|X_{il}|\nonumber\\
\label{boundthm2}&+\sum_{t=1}^d 3^{r_t}B_t\bigg(2^{r_t}\mathbb{E}|X_{il}|\mathbb{E}|W_t|^{r_t}+\frac{2^{r_t}}{n^{r_t/2}}\mathbb{E}|X_{il}X_{it}^{r_t}|+2\mathbb{E}|Z_t|^{r_t+1}\mathbb{E}|X_{il}|\bigg)\bigg]\bigg\}.
\end{align}
\end{theorem}

\begin{theorem}\label{thmsec21}Let $X_{ij}$, $i=1,\ldots,n$, $j=1,\ldots,d$, be defined as in Lemma \ref{noteveng}, but with the additional assumption that $\mathbb{E}|X_{ij}|^{r_k+4}<\infty$ for all $i,$ $j$ and $1\leq k\leq d$.  Suppose $\Sigma$ is non-negative definite and that $\mathbb{E}X_{ij}X_{ik}X_{il}=0$ for all $1\leq i\leq n$ and $1\leq j,k,l\leq d$. Suppose $g\in C_P^4(\mathbb{R}^d)$.  Then, for $h\in C_b^4(\mathbb{R})$,  
\begin{align}&|\mathbb{E}h(g(\mathbf{W}))-\mathbb{E}h(g(\Sigma^{1/2}\mathbf{Z}))| \leq\frac{h_4}{24n^{2}}\sum_{i=1}^n\sum_{j,k,l,m=1}^d\bigg\{A\mathbb{E}|X_{ij}X_{ik}X_{il}X_{im}|\nonumber\\
&+\sum_{t=1}^d 2^{r_t}B_t\bigg(2^{r_t}\mathbb{E}|X_{ij}X_{ik}X_{il}X_{im}|\mathbb{E}|W_t|^{r_t}+\frac{2^{r_t}}{n^{r_t/2}}\mathbb{E}|X_{ij}X_{ik}X_{il}X_{im}X_{it}^{r_t}|\nonumber\\
&+\mathbb{E}|Z_t|^{r_t}\mathbb{E}|X_{ij}X_{ik}X_{il}X_{im}|\bigg)+9|\mathbb{E}X_{ij}X_{ik}|\bigg[A\mathbb{E}|X_{il}X_{im}|\nonumber\\
\label{boundtheorem}&+\sum_{t=1}^d 2^{r_t}B_t\bigg(2^{r_t}\mathbb{E}|X_{il}X_{im}|\mathbb{E}|W_t|^{r_t}+\frac{2^{r_t}}{n^{r_t/2}}\mathbb{E}|X_{il}X_{im}X_{it}^{r_t}|+\mathbb{E}|Z_t|^{r_t}\mathbb{E}|X_{il}X_{im}|\bigg)\bigg]\bigg\}.
\end{align}
\end{theorem}

\begin{proof}Notice that in the proof of Lemma \ref{evennormal} the bound (\ref{nearlyeven}) reduces to $|R_1|+|R_2|+|R_3|+|R_4|$ when $\mathbb{E}X_{ij}X_{ik}X_{il}=0$ for all $1\leq i\leq n$ and $1\leq j,k,l\leq d$.  Therefore the terms involving a multiple of $\mathbb{E}X_{ij}X_{ik}X_{il}$ vanish from the bound (\ref{boundthm2}), and we no longer require that $g$ is even and also only need $g$ to belong to the class $C_P^4(\mathbb{R}^d)$ and $h$ to belong to $C_b^4(\mathbb{R})$.
\end{proof}

\subsection{Relaxing the condition that $g$ is even}

For our application to the rate of convergence of the power divergence statistics we shall need a slight relaxation of the assumption from Theorem \ref{thmsec2} that $g$ is an even function.  Looking back at the proof of Lemma \ref{evennormal}, we see that a crucial step in obtaining the $O(n^{-1})$ rate of Theorem \ref{thmsec2} was the result that $\mathbb{E}\frac{\partial^{3} f}{\partial w_j\partial w_k\partial w_l}(\mathbf{W})$ is of order $n^{-1/2}$ when $g$ is even function satisfying suitable differentiability and boundedness conditions.  We were able to obtain this result by using the fact that $\mathbb{E}\frac{\partial^{3} f}{\partial w_j\partial w_k\partial w_l}(\Sigma^{1/2}\mathbf{Z})=0$ for such $g$.  However, it actually suffices that $\mathbb{E}\frac{\partial^{3} f}{\partial w_j\partial w_k\partial w_l}(\Sigma^{1/2}\mathbf{Z})=O(n^{-1/2})$ in order to obtain a final bound of order $n^{-1}$.  This offers the scope for relaxing the condition that $g$ is an even function.  Suppose $g:\mathbb{R}^d\rightarrow\mathbb{R}$ is such that 
\begin{equation}\label{gdelta}g(\mathbf{w})=a(\mathbf{w})+\delta b(\mathbf{w}),
\end{equation}
where $a:\mathbb{R}^d\rightarrow\mathbb{R}$ and $b:\mathbb{R}^d\rightarrow\mathbb{R}$ satisfy suitable boundedness and differentiability conditions, $a$ is an even function and $0\leq\delta <1$ is a constant that we shall often think of as being `small'.  Through a sequence of lemmas we shall see that, for such a $g$, we have $\mathbb{E}\frac{\partial^{3} f}{\partial w_j\partial w_k\partial w_l}(\Sigma^{1/2}\mathbf{Z})=O(\delta )$.  This will enable us to obtain an extension of Theorem \ref{thmsec2} to the case that $g$ is of the form (\ref{gdelta}).  This theorem will enable us to obtain $O(n^{-1})$ bounds for the rate of convergence of the power divergence statistics for certain values of the index parameter $\lambda$; see Section 3.3. 

We begin by obtaining simple useful formula for the partial derivatives of the test function $h(g(\cdot))$, where $g$ is of the form (\ref{gdelta}).  Before deriving this formula, we state some preliminary results.  The first is a multivariate generalisation of the Fa\`{a} di Bruno formula for $n$-th order derivatives of composite functions, due to \cite{ma}:
\begin{equation}\label{bruno}\frac{\partial^m}{\prod_{j=1}^n\partial w_{i_j}}h(g(\mathbf{w}))=\sum_{\pi\in\Pi}h^{(|\pi|)}(g(\mathbf{w}))\cdot\prod_{B\in\pi}\frac{\partial^{|B|}g(\mathbf{w})}{\prod_{j\in B}\partial w_j},
\end{equation}
where $\pi$ runs through the set $\Pi$ of all partitions of the set $\{1,\ldots,m\}$, the product is over all of the parts $B$ of the partition $\pi$, and $|S|$ is the cardinality of the set $S$.  It is useful to note that the number of partitions of $\{1,\ldots,m\}$ into $k$ non-empty subsets is given by the Stirling number of the second kind ${m\brace k}$ (see \cite{olver}).

We now introduce a class of functions that will be play a similar role to the class of functions $C_P^m(\mathbb{R}^d)$ of Section 2.2.    We say that the function $g:\mathbb{R}^d\rightarrow\mathbb{R}$ belongs to the class $C_{Q,\delta}^m(\mathbb{R}^d)$ if $g$ can be written in the form (\ref{gdelta}), that all $m$-th order partial derivatives of $a$ and $b$ exist and there exists a dominating function $Q:\mathbb{R}^d\rightarrow\mathbb{R}^+$ such that, for all $\mathbf{w}\in\mathbb{R}^d$, the quantities
\begin{equation*}\bigg|\frac{\partial^ka(\mathbf{w})}{\prod_{j=1}^k\partial w_{i_j}}\bigg|^{m/k},\quad \bigg|\frac{\partial^kb(\mathbf{w})}{\prod_{j=1}^k\partial w_{i_j}}\bigg|^{m/k}, \quad |b(\mathbf{w})|\bigg|\frac{\partial^ka(\mathbf{w})}{\prod_{j=1}^k\partial w_{i_j}}\bigg|^{m/k}, \quad 1\leq k\leq m,
\end{equation*}
are all bounded by $Q(\mathbf{w})$.  If $g\in C_{Q,\delta}^m(\mathbb{R}^d)$ then it is easy to see that, for all $\mathbf{w}\in\mathbb{R}^d$, the quantities
\begin{equation*}\bigg|\prod_{B\in\pi}\frac{\partial^{|B|}a(\mathbf{w})}{\prod_{j\in B}\partial w_j}\bigg|, \quad \bigg|\prod_{B\in\pi}\frac{\partial^{|B|}b(\mathbf{w})}{\prod_{j\in B}\partial w_j}\bigg|, \quad \bigg|b(\mathbf{w})\prod_{B\in\pi}\frac{\partial^{|B|}a(\mathbf{w})}{\prod_{j\in B}\partial w_j}\bigg|
\end{equation*}
are all bounded $Q(\mathbf{w})$.  With these inequalities we are able to prove the following lemma.

\begin{lemma}\label{bell lem}Suppose $h\in C_b^m(\mathbb{R})$ and that $g$, defined by $g(\mathbf{w})=a(\mathbf{w})+\delta b(\mathbf{w})$, is in the class $C_{Q,\delta}^m(\mathbb{R}^d)$.  Then, for all $\mathbf{w}\in\mathbb{R}^d$,
\begin{equation}\label{hgdiff}\frac{\partial^m}{\prod_{j=1}^m\partial w_{i_j}}h(g(\mathbf{w}))=\frac{\partial^m}{\prod_{j=1}^m\partial w_{i_j}}h(a(\mathbf{w}))+\delta q(\mathbf{w}),
\end{equation}
where
\begin{equation}\label{qbound}|q(\mathbf{w})|\leq \delta \tilde{h}_mQ(\mathbf{w})
\end{equation}
and $\tilde{h}_m=\sum_{k=1}^m{m\brace k}\big(2^m\|h^{(k)}\|+\|h^{(k+1)}\|\big)$.
\end{lemma}

\begin{proof}By (\ref{bruno}), we have that
\begin{align*}\frac{\partial^m}{\prod_{j=1}^m\partial w_{i_j}}h(g(\mathbf{w}))&=\sum_{\pi\in\Pi}h^{(|\pi|)}(g(\mathbf{w}))\cdot\prod_{B\in\pi}\bigg(\frac{\partial^{|B|}a(\mathbf{w})}{\prod_{j\in B}\partial w_j}+\delta \frac{\partial^{|B|}b(\mathbf{w})}{\prod_{j\in B}\partial w_j}\bigg) \\
&=\sum_{\pi\in\Pi}h^{(|\pi|)}(g(\mathbf{w}))\cdot\prod_{B\in\pi}\frac{\partial^{|B|}a(\mathbf{w})}{\prod_{j\in B}\partial w_j}+r_1(\mathbf{w}),
\end{align*}
where $r_1$ satisfies the crude inequality, for all $\mathbf{w}\in\mathbb{R}^d$, 
\begin{equation*}|r_1(\mathbf{w})|\leq \delta \sum_{k=1}^m{m\brace k}2^m\|h^{(k)}\|Q(\mathbf{w}),
\end{equation*}
as $\delta <1$.  By the mean value theorem we have that
\begin{align*}\frac{\partial^m}{\prod_{j=1}^m\partial w_{i_j}}h(g(\mathbf{w}))&=\sum_{\pi\in\Pi}h^{(|\pi|)}(a(\mathbf{w}))\cdot\prod_{B\in\pi}\frac{\partial^{|B|}a(\mathbf{w})}{\prod_{j\in B}\partial w_j}+r_1(\mathbf{w})+r_2(\mathbf{w})\\
&=\frac{\partial^m}{\prod_{j=1}^m\partial w_{i_j}}h(a(\mathbf{w}))+r_1(\mathbf{w})+r_2(\mathbf{w}),
\end{align*}
where $r_2$ is bounded for all $\mathbf{w}\in\mathbb{R}^d$ by
\begin{align*}|r_2(\mathbf{w})|&\leq \delta \sum_{\pi\in\Pi}\|h^{(|\pi|+1)}\|\bigg|b(\mathbf{w})\prod_{B\in\pi}\frac{\partial^{|B|}a(\mathbf{w})}{\prod_{j\in B}\partial w_j}\bigg|\\
&\leq \delta \sum_{k=1}^m{m\brace k}\|h^{(k+1)}\|Q(\mathbf{w}).
\end{align*}
Summing up the remainders $r_1(\mathbf{w})+r_2(\mathbf{w})$ completes the proof.
\end{proof}

So far, we have imposed no conditions on the dominating function $Q$.  However, from now on, we shall suppose that $Q(\mathbf{w})=A+\sum_{i=1}^dB_i|w_i|^{r_i}$, where $A\geq 0$, $B_1,\ldots,B_d\geq 0$ and $r_1,\ldots,r_d\geq 0$.  Thus, $Q$ takes the same form as $P$ did in Section 2.2.  We shall restrict our attention to such a dominating function in this paper, but we note that we could obtain an analogue of the following lemma for dominating functions that have exponential growth (see \cite{gaunt normal}, Section 2).  However, for our applications, we shall not need such a lemma, so we omit it for space reasons.  

\begin{lemma}Let $m\geq1$ be odd. Suppose that $h\in C_b^m(\mathbb{R})$ and $g\in C_{Q,\delta}^m(\mathbb{R}^d)$.  Let $f$ denote the solution (\ref{mvnsolnh}).  Then
\begin{equation}\label{onway2}\bigg|\mathbb{E}\bigg[\frac{\partial^mf}{\prod_{j=1}^m\partial w_{i_j}}(\Sigma^{1/2}\mathbf{Z})\bigg]\bigg|\leq \delta \frac{\tilde{h}_m}{m}\bigg[A+\sum_{i=1}^d2^{r_i+1}B_i\mathbb{E}|Z_i|^{r_i}\bigg].
\end{equation}
\end{lemma}

\begin{proof}Let $\mathbf{z}_{s,\mathbf{w}}^{\Sigma^{1/2}\mathbf{Z}}=\mathrm{e}^{-s}\mathbf{w}+\sqrt{1-\mathrm{e}^{-2s}}\Sigma^{1/2}\mathbf{Z}$, where $\Sigma^{1/2}\mathbf{Z}$ is an independent copy of $\Sigma^{1/2}\mathbf{Z}\sim\mathrm{MVN}(\mathbf{0},\Sigma)$.  Then, by dominated convergence and Lemma \ref{bell lem}, 
\begin{align*}\frac{\partial^mf(\mathbf{w})}{\prod_{j=1}^m\partial w_{i_j}}&=-\int_0^{\infty}\mathrm{e}^{-ms}\mathbb{E}\bigg[\frac{\partial^m }{\prod_{j=1}^m\partial w_{i_j}}h(g(\mathbf{z}_{s,\mathbf{w}}^{\Sigma^{1/2}\mathbf{Z}'}))\bigg]\,\mathrm{d}s \\
&=-\int_0^{\infty}\mathrm{e}^{-ms}\mathbb{E}\bigg[\frac{\partial^m }{\prod_{j=1}^m\partial w_{i_j}}h(a(\mathbf{z}_{s,\mathbf{w}}^{\Sigma^{1/2}\mathbf{Z}'}))\bigg]\,\mathrm{d}s \\
&\quad-\delta \int_0^\infty \mathrm{e}^{-ms}\mathbb{E}q(\mathbf{z}_{s,\mathbf{w}}^{\Sigma^{1/2}\mathbf{Z}'})\,\mathrm{d}s,
\end{align*}
where $q$ is defined as per equation (\ref{hgdiff}).  Evaluating both sides at the random variable $\Sigma^{1/2}\mathbf{Z}$ and taking expectations gives that
\begin{align}\mathbb{E}\bigg[\frac{\partial^mf(\Sigma^{1/2}\mathbf{Z})}{\prod_{j=1}^m\partial w_{i_j}}\bigg]&=-\int_0^{\infty}\mathrm{e}^{-ms}\mathbb{E}\bigg[\frac{\partial^m }{\prod_{j=1}^m\partial w_{i_j}}h(a(\mathbf{z}_{s,\Sigma^{1/2}\mathbf{Z}}^{\Sigma^{1/2}\mathbf{Z}'}))\bigg]\,\mathrm{d}s\nonumber\\
\label{vanish}&\quad-\delta \int_0^\infty \mathrm{e}^{-ms}\mathbb{E}q(\mathbf{z}_{s,\Sigma^{1/2}\mathbf{Z}}^{\Sigma^{1/2}\mathbf{Z}'})\,\mathrm{d}s.
\end{align}
Now, $\mathbf{z}_{s,\Sigma^{1/2}\mathbf{Z}}^{\Sigma^{1/2}\mathbf{Z}'}=\mathrm{e}^{-s}\Sigma^{1/2}\mathbf{Z}+\sqrt{1-\mathrm{e}^{-2s}}\Sigma^{1/2}\mathbf{Z}'\stackrel{\mathcal{D}}{=}\Sigma^{1/2}\mathbf{Z}$, which is equal in distribution to $-\Sigma^{1/2}\mathbf{Z}$.  Since $a$ is an even function and $m$ is odd, it therefore follows that the first integral on the right-hand side of (\ref{vanish}) is equal to 0, and so
\begin{align*}\mathbb{E}\bigg[\frac{\partial^mf(\Sigma^{1/2}\mathbf{Z})}{\prod_{j=1}^m\partial w_{i_j}}\bigg]=-\delta \int_0^\infty \mathrm{e}^{-ms}\mathbb{E}q(\mathbf{z}_{s,\Sigma^{1/2}\mathbf{Z}}^{\Sigma^{1/2}\mathbf{Z}'})\,\mathrm{d}s,
\end{align*}
and thus, by (\ref{qbound}),
\begin{align}\label{abovee}\bigg|\mathbb{E}\bigg[\frac{\partial^mf(\Sigma^{1/2}\mathbf{Z})}{\prod_{j=1}^m\partial w_{i_j}}\bigg]\bigg|\leq\delta \tilde{h}_m\int_0^\infty \mathrm{e}^{-ms}\mathbb{E}Q(\mathbf{z}_{s,\Sigma^{1/2}\mathbf{Z}}^{\Sigma^{1/2}\mathbf{Z}'})\,\mathrm{d}s.
\end{align}
It was shown in \cite{gaunt normal} (see Lemma 2.3 and Corollary 2.2 of that work) that, for all $\mathbf{w}\in\mathbb{R}^d$,
\begin{equation*}\int_0^\infty \mathrm{e}^{-ms}\mathbb{E}Q(\mathbf{z}_{s,\Sigma^{1/2}\mathbf{Z}}^{\Sigma^{1/2}\mathbf{Z}'})\,\mathrm{d}s\leq \frac{1}{m}\bigg[A+B\sum_{i=1}^d2^{r_i}\big(|w_i|^{r_i}+\mathbb{E}|Z_i|^{r_i}\big)\bigg].
\end{equation*}
Applying this inequality to (\ref{abovee}) gives that
\begin{align*}\bigg|\mathbb{E}\bigg[\frac{\partial^mf(\Sigma^{1/2}\mathbf{Z})}{\prod_{j=1}^m\partial w_{i_j}}\bigg]\bigg|&\leq\delta \frac{\tilde{h}_m}{m}\bigg[A+B\sum_{i=1}^d2^{r_i}\big(\mathbb{E}|Z_i|^{r_i}+\mathbb{E}|Z_i|^{r_i}\big)\bigg]\\
&=\delta \frac{\tilde{h}_m}{m}\bigg[A+\sum_{i=1}^d2^{r_i+1}B_i\mathbb{E}|Z_i|^{r_i}\bigg],
\end{align*}
as required.
\end{proof}

\begin{theorem}\label{thmsec3}Let $X_{ij}$, $i=1,\ldots,n$, $j=1,\ldots,d$, be defined as in Lemma \ref{noteveng}, but with the additional assumption that $\mathbb{E}|X_{ij}|^{r_k+4}<\infty$ for all $i,$ $j$ and $1\leq k\leq d$.  Suppose $\Sigma$ is non-negative definite and that $g\in C_{Q,n^{-1/2}}^6(\mathbb{R}^d)$.  Then, for $h\in C_b^6(\mathbb{R})$,  
\begin{align*}&|\mathbb{E}h(g(\mathbf{W}))-\mathbb{E}h(g(\Sigma^{1/2}\mathbf{Z}))|\\
&\leq M+\frac{\tilde{h}_3}{6n^2}\sum_{i=1}^n\sum_{j,k,l=1}^d|\mathbb{E}X_{ij}X_{ik}X_{il}|\bigg[A+\sum_{t=1}^d2^{r_t+1}B_t\mathbb{E}|Z_t|^{r_t}\bigg],
\end{align*}
where $M$ is defined as in Theorem \ref{thmsec2}.
\end{theorem}

\begin{proof}Theorem \ref{thmsec2} was obtained by applying Lemma \ref{evennormal} together with the bounds of (\ref{cor28}).  Examining the proof of Lemma \ref{evennormal} (particularly equation (\ref{nearlyeven})), we see that we can obtain a bound for the quantity $|\mathbb{E}h(g(\mathbf{W}))-\mathbb{E}h(g(\Sigma^{1/2}\mathbf{Z}))|$ that is the bound $M$ of Theorem \ref{thmsec2} plus the additional term
\begin{equation}\label{onway}\frac{1}{2n^{3/2}}\sum_{i=1}^n\sum_{j,k,l=1}^d|\mathbb{E}X_{ij}X_{ik}X_{il}|\bigg|\mathbb{E}\frac{\partial^3f}{\partial w_j\partial w_k\partial w_l}(\Sigma^{1/2}\mathbf{Z})\bigg|,
\end{equation}
which arises because it is no longer assumed that $g$ is an even function.  Applying inequality (\ref{onway2}), with $\delta =n^{-1/2},$ to bound (\ref{onway}) yields the desired $O(n^{-1})$ bound.
\end{proof}

\section{Application to Freidman's chi-square and the power divergence statistics}

In this section, we consider the application of the approximation theorems of Section 2 to the statistics for complete block designs that were introduced in Section 1.1.

\subsection{Friedman's statistic}

We begin be observing the Friedman's statistic falls into the class of statistics covered by Theorem \ref{thmsec21}.  Let $W_j=\frac{1}{\sqrt{n}}\sum_{i=1}^n X_{ij}$,  where $X_{ij}=\frac{\sqrt{12}}{\sqrt{r(r+1)}}\big(\pi_i(j)-\frac{r+1}{2}\big)$, and under the null hypothesis, for fixed $j$, the collection of random variables $\pi_1(j),\ldots,\pi_n(j)$ are i.i.d$.$ with uniform distribution on $\{1,\ldots,r\}$.   Friedman's statistic is given by
\begin{equation*}F_r=\sum_{j=1}^rW_j^2,
\end{equation*}
and is asymptotically $\chi_{(r-1)}^2$ distributed under the null hypothesis.  

By the central limit theorem, for any $j$, we have that $W_j$ converges in distribution to a mean zero normal random variables as $n\rightarrow\infty$.  However, the $W_j$ are not independent (see Lemma \ref{robcov} for the (non-negative definite) covariance matrix $\Sigma_{\mathbf{W}}$), and we have that $\mathbf{W}=(W_1,\ldots,W_r)^T\stackrel{\mathcal{D}}{\rightarrow}\mathrm{MVN}(\mathbf{0},\Sigma_{\mathbf{W}})$ as $n\rightarrow\infty$.  However, for a fixed $j$, the random variables $X_{1,j},\ldots,X_{n,j}$ are independent because it is assumed that trials are independent.  Also, we can write $\chi^2=g(\mathbf{W})$, where $g(\mathbf{w})=\sum_{j=1}^mw_j^2$.  The function $g:\mathbb{R}^r\rightarrow\mathbb{R}$ is an even function with partial derivatives of polynomial growth.  Finally, since the $X_{ij}$ have finite support, their moments of any order exist.  Therefore, Friedman's statistic falls into the framework of Theorem \ref{thmsec2}.  However, since the distribution of the random variables $X_{ij}$ is symmetric about 0, it follows that $\mathbb{E}X_{ij}X_{ik}X_{il}=0$ for all $1\leq i\leq n$ and $1\leq j,k,l\leq d$.  Thus, we can in fact apply Theorem \ref{thmsec21} to bound the rate of convergence of Friedman's statistic.  We present an explicit bound in Theorem \ref{thm1}, but before stating the theorem we note the following lemma. 

\begin{lemma}\label{robcov} The (non-negative definite) covariance matrix of $\mathbf{W}$, denoted by $\Sigma_{\mathbf{W}} = (\sigma_{jk})$, has entries
\begin{equation*} \sigma_{jj} = \frac{r-1}{r} \quad \text{and} \quad \sigma_{jk} = -\frac{1}{r} \quad (j \neq k).
 \end{equation*}
\end{lemma}
\begin{proof}For fixed $j$, $\pi_1(j),\ldots,\pi_n(j)$ are independent $\mathrm{Unif}\{1,\ldots,r\}$ random variables, and therefore
\[\sigma_{jj}=\mathrm{Var}W_j=\frac{12}{r(r+1)n}\sum_{i=1}^n\mathrm{Var}\pi_i(j)=\frac{12}{r(r+1)n}\times n\times\frac{r^2-1}{12}=\frac{r-1}{r}.\]
Suppose now that $j\not=k$.  Since $\sum_{j=1}^rW_j=0$, we have
\[0=\mathbb{E}\bigg[W_j\sum_{l=1}^rW_l\bigg]=\mathbb{E}W_j^2+\sum_{l\not= j}\mathbb{E}W_jW_l=\mathbb{E}W_j^2+(r-1)\mathbb{E}W_jW_k,\]
where we used that the $W_j$ are identically distributed to obtain the final equality.  On rearranging, and using that $\mathbb{E}W_j^2=\frac{r-1}{r}$, we have that $\mathbb{E}W_jW_k=-\frac{1}{r}$ for $j\not=k$.  As $\mathbb{E}W_j=0$, it follows that $\sigma_{jk}=\mathrm{Cov}(W_j,W_k)=\mathbb{E}W_jW_k=-\frac{1}{r}$, for $j\not=k$, as required.
\end{proof}

\begin{theorem}\label{thm1}Suppose $r\geq 2$.  Then, for $h\in C_b^4(\mathbb{R}^+)$, 
\begin{equation}\label{boundthm1}|\mathbb{E}h(F_r)-\chi_{(r-1)}^2h|\leq 10797r^5n^{-1}h_4,
\end{equation}
where $\chi_{(r-1)}^2h$ denotes the expectation of $h(Y)$ for $Y\sim \chi_{(r-1)}^2$.
\end{theorem}

\begin{remark}

\begin{enumerate}

\item The bound (\ref{boundthm1}) is of order $n^{-1}$, which is the fastest rate of convergence in the literature for Friedman's statistic.  However, the numerical constants and the dependence of the bound on $r$ are far from optimal.  This is the price we pay for deriving the bound by applying the more general bound of Theorem \ref{thmsec21}.  In proving Theorem \ref{thmsec21}, we used local couplings to deal with the dependence structure, and this approach enabled us to obtain a $O(n^{-1})$ bound for a class of statistics for complete block designs.  However, for Friedman's statistic, local couplings are quite crude, and their use leads to a bound with a poor dependence on $r$.  A direction for future research is to obtain a $O(n^{-1})$ bound with a better dependence on $r$. 

\item As using the Theorem \ref{thmsec21} to obtain a $O(n^{-1})$ bound for Friedman's statistic will lead to one with a far from optimal dependence on $r$ and large numerical constants, we make use of a number of crude inequalities to derive when applying Theorem \ref{thmsec21} to derive our bound.  This simplifies the calculations, but still allows us to obtain the desired $O(n^{-1})$ rate.

\item We could also apply Theorem \ref{theoremsec21} to derive a $O(r^3n^{-1/2})$ bound for Friedman's statistic, which may be preferable when the numerical constants are large compared to $n$.  Note that we cannot apply Theorem \ref{theoremsec22} because the covariance matrix $\Sigma_{\mathbf{S}}$ is not positive-definite. 

\end{enumerate}

\end{remark}  

\medskip
\noindent{\emph{Proof of Theorem \ref{thm1}.}}
As described above, Friedman's statistic falls into the framework of Theorem \ref{thmsec21}, so to derive the bound we need to find a suitable dominating function for $g(\mathbf{w})=\sum_{j=1}^mw_j^2$ and then bound the expectations that appear in the bound (\ref{boundtheorem}).  For all $k=1,\ldots,d$ we have that $\frac{\partial g(\mathbf{w})}{\partial w_k}=2w_k$, $\frac{\partial^2g(\mathbf{w})}{\partial w_k^2}(\mathbf{w})=2$ and all other derivatives are equal to 0.  Now, $\big|\frac{\partial g(\mathbf{w})}{\partial w_k}\big|^4=16w_k^4$ and $\big|\frac{\partial^2 g(\mathbf{w})}{\partial w_k^2}\big|^2=4$, meaning that we can take $P(\mathbf{w})=4+16\sum_{j=1}^rw_j^4$ as our dominating function.  We therefore apply the bound (\ref{boundtheorem}) with $A=4$, $B_1=\ldots=B_r=16$ and $r_1=\ldots =r_r=4$.

Now, let $X$ be a random variable that has the same distribution as the $X_{ij}$.  We shall need some formulas for the moments of $X$, which can be computed from the following sum:
\begin{equation*}\mathbb{E}X^m=\bigg(\frac{12}{r(r+1)}\bigg)^{m/2}\sum_{j=1}^r\bigg(j-\frac{r+1}{2}\bigg)^m.
\end{equation*}
In particular,
\begin{align*}\mathbb{E}X^2&=\frac{r-1}{r}\leq1, \quad\quad \mathbb{E}X^4=\frac{144}{r^2(r+1)^2}\cdot\frac{1}{240}(r^2-1)(3r^2-7)\leq\frac{9}{5}, \\
\mathbb{E}X^6&=\frac{12^3}{r^3(r+1)^3}\cdot\frac{1}{1344}(r^2-1)(3r^4-18r^2+31)\leq\frac{27}{7}, \\
\mathbb{E}X^8&=\frac{12^4}{r^4(r+1)^4}\cdot\frac{1}{33792}(r^2-1)(r^2-5)(3r^6-37r^4+225r^2-511)\leq 9.
\end{align*}
With these inequalities and H\"{o}lder's inequality we can bound the following terms that arise in (\ref{boundtheorem}) for all $1\leq i\leq n$ and $1\leq j,k,l,m,t\leq d$:
\begin{eqnarray*}|\mathbb{E}X_{ij}X_{ik}|\leq\mathbb{E}|X_{ij}X_{ik}|&\leq& (\mathbb{E}X^2)^{1/2}\leq1, \\
\mathbb{E}|X_{ij}X_{ik}X_{il}X_{im}|&\leq&(\mathbb{E}X^4)^{1/4}\leq(9/5)^{1/4}, \\
\mathbb{E}|X_{ij}X_{ik}X_{il}^4|&\leq&(\mathbb{E}X^6)^{1/6}\leq(27/7)^{1/6}, \\
\mathbb{E}|X_{ij}X_{ik}X_{il}X_{im}X_{it}^4|&\leq&(\mathbb{E}X^8)^{1/8}\leq 9^{1/8}.
\end{eqnarray*}
We also have that, for all $1\leq j\leq d$,
\begin{align*}\mathbb{E}W_j^4=\frac{3(n-1)}{n}(\mathbb{E}X^2)^2+\mathbb{E}X^4\leq 3+\frac{9}{5n}\leq\frac{24}{5},
\end{align*}
as $\mathbb{E}X=0$ and $n\geq1$.  Finally, $\mathbb{E}Z_j^4=3\big(\frac{r-1}{r}\big)^2<3$ for $Z_j\sim N(0,\sigma_{jj})$.  Plugging these inequalities into the bound (\ref{boundtheorem}) and simplifying using that $r,n\geq1$ gives that
\begin{align*}&|\mathbb{E}h(F_r)-\chi_{(r-1)}^2h|\\
&\leq\frac{r^4h_4}{24n}\bigg[4\bigg(\frac{9}{5}\bigg)^{1/4}+2^8r\bigg(16\cdot\bigg(\frac{9}{5}\bigg)^{1/4}\cdot\frac{24}{5}+\frac{16\cdot 9^{1/8}}{n^2}+3\cdot\bigg(\frac{9}{5}\bigg)^{1/4}\bigg)\\
&\quad+9\bigg(4+2^8r\bigg(16\cdot\frac{24}{5}+\frac{16}{n^2}\bigg(\frac{27}{7}\bigg)^{1/6}+3\bigg)\bigg)\bigg] \\
&\leq 10797r^5n^{-1}h_4,
\end{align*} 
which is our final bound. \hfill $\square$

\subsection{Pearson's statistic}

Here, we consider the application of Theorem \ref{thmsec2} to Pearson's statistic.  Recall that Pearson's statistic is given by
\begin{equation*} \chi^2 = \sum_{j=1}^r \frac{(U_j - n p_j)^2}{n p_j}, 
\end{equation*}
and is asymptotically $\chi_{(r-1)}^2$ distributed provided $np_*\rightarrow\infty$, where $p_*=\min_{1\leq j\leq r}p_j$.  The cell counts $U_j\sim\mathrm{Bin}(n,p_j)$, $1\leq j\leq r$, are dependent random variables that satisfy $\sum_{j=1}^rU_j=n$.  Now, let $I_{ij}\sim \mathrm{Ber}(p_j)$ denote the indicator that the $i$-th trial falls in the $j$-th cell.  Then letting $X_{ij}=\frac{I_{ij}-p_j}{\sqrt{p_j}}$ and $W_j=\frac{1}{\sqrt{n}}\sum_{i=1}^nX_{ij}$, we can write
\begin{equation*}\chi^2=g(\mathbf{W})=\sum_{j=1}^rW_j^2,
\end{equation*}
where $g(\mathbf{w})=\sum_{j=1}^rw_j^2$.  As was the case for Friedman's statistic, the function $g$ is even and has derivatives of polynomial growth.  Also, because trials are assumed to be independent, the random variables $X_{1,j},\ldots,X_{n,j}$ are independent for fixed $j$.  The covariance matrix of $\mathbf{W}=(W_1,\ldots,W_r)^T$ is non-negative definite, with entries $\sigma_{jj}=1-p_j$ and $\sigma_{jk}=-\sqrt{p_jp_k}$, $j\not=k$ (see \cite{gaunt chi square}).  It is therefore the case that Pearson's statistics falls within the class of statistics covered by Theorem \ref{thmsec2}.  However, unlike, Friedman's statistic, we cannot apply Theorem \ref{thmsec21} to Pearson's statistic.  This is because $\mathbb{E}X_{ij}X_{ik}X_{il}\not=0$ in general.

We can apply Theorem \ref{thmsec2} to Pearson's statistic as follows.  As was the case for Friedman's statistic, for all $k=1,\ldots,d$ we have that $\frac{\partial g(\mathbf{w})}{\partial w_k}=2w_k$, $\frac{\partial^2g(\mathbf{w})}{\partial w_k^2}(\mathbf{w})=2$ and all other derivatives are equal to 0.   However, because we are applying Theorem \ref{thmsec2} instead of Theorem \ref{thmsec21}, we need a different dominating function.  We have that $\big|\frac{\partial g(\mathbf{w})}{\partial w_k}\big|^6=64w_k^6$ and $\big|\frac{\partial^2 g(\mathbf{w})}{\partial w_k^2}\big|^3=8$, meaning that we can take $P(\mathbf{w})=8+64\sum_{j=1}^rw_j^6$ as our dominating function.  We can therefore bound the quantity $|\mathbb{E}h(\chi^2)-\chi_{(r-1)}^2|$ by applying   the bound (\ref{boundthm2}) with $A=8$, $B_1=\ldots=B_r=64$ and $r_1=\ldots =r_r=6$.  We do not compute this bound, but note that we can obtain one of the form
\begin{equation}\label{formchi}|\mathbb{E}h(\chi^2)-\chi_{(r-1)}^2|\leq Cn^{-1}(h_4+h_6),
\end{equation}
where $C$ is a constant depending on $r$ and $p_1,\ldots,p_r$, but not $n$.  We do not explicitly find such a $C$ because a superior upper bound of the form $Km(np_*)^{-1}\sum_{k=1}^5\|h^{(k)}\|$ has already been obtained by \cite{gaunt chi square}.  This bound holds for a weaker class of test functions than (\ref{formchi}) and has a much better dependence on $r$ and $p_1,\ldots,p_r$ than a bound that would result from an application of Theorem \ref{thmsec2}.  That the bound of \cite{gaunt chi square} outperforms ours is perhaps to be expected, given that their approach was target specifically at Pearson's statistic rather than the general class of statistics that are considered in this paper.  

\subsection{The power divergence family of statistics}

Recall that the power divergence statistic with index $\lambda\in\mathbb{R}$ is given by
\begin{equation}\label{cbshbac}T_\lambda=\frac{2}{\lambda(\lambda+1)}\sum_{j=1}^rU_j\bigg[\bigg(\frac{U_j}{np_j}\bigg)^\lambda-1\bigg].
\end{equation}
Letting $\mathbf{W}$ be defined as in Section 3.2, we can write (\ref{cbshbac}) as
\begin{equation}\label{geqndef}T_\lambda(\mathbf{W})=g(\mathbf{W})=\frac{2}{\lambda(\lambda+1)}\bigg[\sum_{j=1}^rnp_j\bigg(1+\frac{W_j}{\sqrt{np_j}}\bigg)^{\lambda+1}-n\bigg],
\end{equation}
since $W_j=\frac{U_j-np_j}{\sqrt{np_j}}$ and $\sum_{j=1}^rU_j=n$.  For general $\lambda$, the function $g$, defined as per equation (\ref{geqndef}), is not an even function; the notable exception being the case $\lambda=1$, which  corresponds to Pearson's statistic.  Therefore, for $\lambda\not=1$, we cannot apply Theorems \ref{thmsec2} and \ref{thmsec21} to obtain $O(n^{-1})$ bounds for the rate of convergence of the statistic $T_\lambda$ to its limiting $\chi_{(r-1)}^2$ distribution.  However, for certain values of $\lambda$, the statistic $T_\lambda$ falls into the class of statistics covered by Theorem \ref{thmsec3}.  Let us now see why this is the case.  We can write
\begin{equation}\label{tformula}T_\lambda(\mathbf{W})=\sum_{j=1}^rW_j^2+R(\mathbf{W}),
\end{equation}
where
\begin{align*}R(\mathbf{W})&=\frac{2}{\lambda(\lambda+1)}\sum_{j=1}^r\bigg[np_j\bigg(1+\frac{W_j}{\sqrt{np_j}}\bigg)^{\lambda+1}\\
&\quad-np_j-(\lambda+1)\sqrt{np_j}W_j-\frac{\lambda(\lambda+1)}{2}W_j^2\bigg],
\end{align*}
as $\sum_{j=1}^rp_j=1$ and $\sum_{j=1}^r\sqrt{p_j}W_j=\sum_{j=1}^r(U_j-np_j)=0$.  Equation (\ref{tformula}) tells us that $T_\lambda(\mathbf{w})$ is of the form (\ref{gdelta}), in the sense that it can be expressed as the sum of an even term $\sum_{j=1}^rw_j^2$ (which corresponds to Pearson's statistic) and a `small' term $R(\mathbf{w})$.  We can argue informally to see that this term is small.  Recall the binomial series formula $(1+x)^{\alpha}=\sum_{k=0}^\infty\frac{(-\alpha)_k}{k!}(-x)^k$ which is valid for $|x|<1$, and the Pochhammer symbol is defined by $(\beta)_n=\beta(\beta+1)\cdots(\beta+n-1)$.  Then, provided $|w_j|<\sqrt{np_j}$ for all $1\leq j\leq r$, we can use the binomial series formula to obtain that
\begin{align}R(\mathbf{w})&=\frac{2}{\lambda(\lambda+1)}\sum_{j=1}^rnp_j\bigg[\sum_{k=0}^{\infty}\frac{(-\lambda-1)_k}{k!}\bigg(\frac{-w_j}{\sqrt{np_j}}\bigg)^k\nonumber\\
&\quad-1-\frac{\lambda+1}{\sqrt{np_j}}w_j-\frac{\lambda(\lambda+1)}{2}\frac{w_j^2}{np_j}\bigg]\nonumber \\
\label{seriesexp}&=\frac{2}{\lambda(\lambda+1)}\sum_{j=1}^r\sum_{k=3}^\infty\frac{(-\lambda-1)_k(-w_j)^k}{k!(np_j)^{k/2-1}},
\end{align}
which is $O(n^{-1/2})$ as $n\rightarrow\infty$.  Whilst the series expansion (\ref{seriesexp}) shows that $R(\mathbf{w})$ is $O(n^{-1/2})$, provided $|w_j|<\sqrt{np_j}$ for all $1\leq j\leq r$, we still need to argue carefully to apply Theorem \ref{thmsec3} to obtain a bound for the rate of convergence of $T_\lambda$; in fact, as we shall now see, the theorem cannot be directly applied for all $\lambda\in\mathbb{R}$.  We shall now prove the following theorem and end by discussing the corresponding conjecture.

\begin{theorem}\label{thmpd}Let $r\geq 2$ and suppose that $\lambda$ is a positive integer or any real number greater than 5.  Then, for $h\in C_b^6(\mathbb{R}^+)$, 
\begin{equation}\label{boundpds}|\mathbb{E}h(T_\lambda(\mathbf{W}))-\chi_{(r-1)}^2h|\leq C(\lambda,p_1,\ldots,p_r)r^7n^{-1}(h_6+\tilde{h}_3),
\end{equation}
where $C(\lambda,p_1,\ldots,p_r)$ is a constant involving $\lambda$ and $p_1,\ldots,p_r$, but not $n$ and $r$.
\end{theorem}

\begin{conjecture}The $O(n^{-1})$ rate of Theorem \ref{thmpd}  holds for all $\lambda\in\mathbb{R}$.
\end{conjecture}

\medskip
\noindent{\emph{Proof of Theorem \ref{thmpd}.}}
Let $a(\mathbf{w})=\sum_{j=1}^rw_j^2$ and $b(\mathbf{w})=\sqrt{n}R(\mathbf{w})$, so that $T_\lambda(\mathbf{w})=a(\mathbf{w})+n^{-1/2}b(\mathbf{w})$.  To apply Theorem \ref{thmsec3}, we need to find a polynomial dominating function $Q$ such that $T_\lambda\in C_{Q,n^{-1/2}}^6(\mathbb{R}^r)$.  We shall now show that we can find such a dominating function if either $\lambda$ is a positive integer or any real number greater than 5.   We have that $\frac{\partial a(\mathbf{w})}{\partial w_k}=2w_k$ and $\frac{\partial^2a(\mathbf{w})}{\partial w_k^2}=2$ and all other derivatives are equal to 0.  Recall that
\begin{align*}b(\mathbf{w})&=\frac{2\sqrt{n}}{\lambda(\lambda+1)}\sum_{j=1}^rnp_j\bigg[\bigg(1+\frac{w_j}{\sqrt{np_j}}\bigg)^{\lambda+1}\\
&\quad-1-(\lambda+1)\frac{w_j}{\sqrt{np_j}}-\frac{\lambda(\lambda+1)}{2}\frac{w_j^2}{np_j}\bigg].
\end{align*} 

If $\lambda$ is a positive integer greater than 1 ($b(\mathbf{w})=0$ when $\lambda=1$), then
\begin{align*}b(\mathbf{w})=\frac{2\sqrt{n}}{\lambda(\lambda+1)}\sum_{j=1}^rnp_j\bigg\{\binom{\lambda+1}{3}\bigg(\frac{w_j}{\sqrt{np_j}}\bigg)^3+\cdots+\bigg(\frac{w_j}{\sqrt{np_j}}\bigg)^{\lambda+1}\bigg\}.
\end{align*}
Thus, when $\lambda$ is a positive integer, the function $b$ is a $O(1)$ polynomial in $r$ variables, and all its partial derivatives of any order are $O(1)$ polynomials or simply 0.  As the derivatives of $a$ are also $O(1)$ polynomials or 0, it follows that there exists a $O(1)$ polynomial dominating function of the form $Q(\mathbf{w})=A+\sum_{i=1}^rB_i|w_i|^{r_i}$ such that $T_\lambda\in C_{Q,n^{-1/2}}^6(\mathbb{R}^r)$.  The random variables $X_{ij}$ have bounded support and all their moments exist and are $O(1)$ (with respect to $n$) and also all moments of the $W_j$ are $O(1)$.  Thus, applying Theorem \ref{thmsec3} for the case $\lambda\in\mathbb{Z}^+$ yields a bound of the form (\ref{boundpds}).  To obtain an explicit form for $C(\lambda,p_1,\ldots,p_r)$, we would need to compute the relevant expectations involving the $X_{ij}$ and $W_j$.

Now let us suppose that $\lambda$ is not necessarily an integer.  For $1\leq k\leq r$, the derivatives of $b$ are given by
\begin{align*}\frac{\partial b(\mathbf{w})}{\partial w_k}&=\frac{2n\sqrt{p_k}}{\lambda}\bigg[\bigg(1+\frac{w_k}{\sqrt{np_k}}\bigg)^\lambda-1-\frac{\lambda w_k}{\sqrt{np_k}}\bigg], \\
\frac{\partial^2b(\mathbf{w})}{\partial w_k^2}&=2\sqrt{n}\bigg[\bigg(1+\frac{w_k}{\sqrt{np_k}}\bigg)^{\lambda-1}-1\bigg], \\
\frac{\partial^mb(\mathbf{w})}{\partial w_k^m}&=\frac{2(\lambda-1)\cdots(\lambda+2-m)\sqrt{n}}{(np_k)^{m/2-1}}\bigg(1+\frac{w_k}{\sqrt{np_k}}\bigg)^{\lambda+1-m}, \quad m\geq3,
\end{align*}
and all other derivatives are equal to 0.   Recall that to apply Theorem \ref{thmsec3}, we need to find a polynomial dominating function $Q$ such that $T_\lambda\in C_{Q,n^{-1/2}}^6(\mathbb{R}^r)$.  This is not possible if $\lambda$ is strictly less than 5 and not a positive integer, because in such cases there will be singularity at $w_k=-\sqrt{np_k}$ (note that this value is in the support of $W_k$) for at least one of the derivatives of order 6 or less.  Since we have already covered the positive integer case, we suppose $\lambda\geq5$.

Let us first note the following crude inequalities.  Let $\alpha\geq3$.  Then, for any $x\geq-1$,
\begin{eqnarray*}\big|(1+x)^\alpha-1-\alpha x-\tfrac{1}{2}\alpha(\alpha-1)x^2\big|&\leq& 2^\alpha(|x|^3+|x|^\alpha), \\
\big|(1+x)^\alpha-1-\alpha x\big|&\leq& 2^\alpha(x^2+|x|^\alpha), \\
\big|(1+x)^\alpha-1\big|&\leq& 2^\alpha(|x|+|x|^\alpha), \\
(1+x)^\alpha &\leq& 2^\alpha(1+|x|^\alpha). 
\end{eqnarray*}
With these inequalities, we obtain that, for $w_k\geq-\sqrt{np_k}$ for all $k=1,\ldots,r$,
\begin{align*}|b(\mathbf{w})|&\leq\frac{2^{\lambda+2}}{\lambda(\lambda+1)}\sum_{j=1}^rn^{3/2}p_j\bigg(\bigg|\frac{w_j}{\sqrt{np_j}}\bigg|^3+\bigg|\frac{w_j}{\sqrt{np_j}}\bigg|^{\lambda+1}\bigg), \\
\bigg|\frac{\partial b(\mathbf{w})}{\partial w_k}\bigg|&\leq\frac{2^{\lambda+1}n\sqrt{p_k}}{\lambda}\bigg(\bigg(\frac{w_k}{\sqrt{np_k}}\bigg)^2+\bigg|\frac{w_k}{\sqrt{np_k}}\bigg|^{\lambda}\bigg), \\
\bigg|\frac{\partial^2 b(\mathbf{w})}{\partial w_k^2}\bigg|&\leq 2^\lambda\sqrt{n}\bigg(\bigg|\frac{w_k}{\sqrt{np_k}}\bigg|+\bigg|\frac{w_k}{\sqrt{np_k}}\bigg|^{\lambda-1}\bigg), \\
\bigg|\frac{\partial^mb(\mathbf{w})}{\partial w_k^m}\bigg|&\leq \frac{2^{\lambda+2-m}(\lambda-1)\cdots(\lambda+2-m)\sqrt{n}}{(np_k)^{m/2-1}}\bigg(1+\bigg|\frac{w_k}{\sqrt{np_k}}\bigg|^{\lambda+m-1}\bigg),
\end{align*}
for $3\leq m\leq 6$.  Thus, if $\lambda\geq5$, all derivatives up to sixth order of $b$ are bounded above by a $O(1)$ polynomial in the region $w_k\geq-\sqrt{np_k}$, $k=1,\ldots,r$.  We can therefore argue as we did for the $\lambda\in\mathbb{Z}^+$ case to conclude that when $\lambda\geq5$ there exists a $O(1)$ polynomial dominating function $Q(\mathbf{w})=A+\sum_{i=1}^rB_i|w_i|^{r_i}$ such that $T_\lambda\in C_{Q,n^{-1/2}}^6(\mathbb{R}^r)$.  That the bound is of the form (\ref{boundpds}) also follows by the same argument.  Having dealt with the cases $\lambda\in\mathbb{Z}^+$ and $\lambda\geq5$, we have completed the proof.   \hfill $\square$

\begin{remark}

\begin{enumerate}

\item To prove Theorem \ref{thmpd}, it sufficed to prove that there exists a $O(1)$ polynomial dominating function $Q$ such that $T_\lambda\in C_{Q,n^{-1/2}}^6(\mathbb{R}^r)$ for $\lambda\in\mathbb{Z}^+$ and $\lambda\geq5$.  However, it is easy to verify that we could have taken the following dominating function:
\begin{align}\label{q1}Q(\mathbf{w})&=A^\lambda+\sum_{j=1}^rB_j^\lambda\bigg(|w_j|^{12}+\bigg|\frac{w_j}{\sqrt{np_j}}\bigg|^{6\lambda}\bigg)\\
\label{q2}&\leq \tilde{A}^\lambda+\sum_{j=1}^r\tilde{B}_j^\lambda|w_j|^{6\lambda},
\end{align}
where $A^\lambda$, $\tilde{A}^\lambda$, $\tilde{B}_1^\lambda,\ldots,\tilde{B}_r^\lambda$ and $B_1^\lambda,\ldots,B_r^\lambda$ are $O(1)$ non-negative constants involving only $\lambda$.  We could directly apply Theorem \ref{thmsec3} with the dominating function (\ref{q2}) to obtain a bound for $|\mathbb{E}h(T_\lambda(\mathbf{W}))-\chi_{(r-1)}^2h|$.  Alternatively, we that we could easily derive a variant of Theorem \ref{thmsec3} for dominating functions of the form (\ref{q1}).  This would result the same bound as that given in Theorem \ref{thmsec3}, but with an additional $O(n^{-3\lambda})$ term.  Thus, the leading order term (in $n$) of our bound for $|\mathbb{E}h(T_\lambda(\mathbf{W}))-\chi_{(r-1)}^2h|$ would result from applying Theorem \ref{thmsec3} with the dominating function $A^\lambda+\sum_{j=1}^rB_j^\lambda|w_j|^{12}$.  The bound would then be computed by bounding the appropriate expectations involving the $X_{ij}$ and $W_j$.  Since $r_1=\ldots=r_r=12$, the values of these expectations would not depend on $\lambda$.  Therefore the performance of the bound based on $p_1,\ldots,p_r$ would remain the same for any $r\in\mathbb{Z}$ or $r\geq5$.

\item Given that the application of Theorem \ref{thmsec2} to Pearson's statistic resulted in a bound with a poor dependence on $p_1,\ldots,p_r$ and $r$ compared to the existing bound of \cite{gaunt chi square}, we would also expect that the application of Theorem \ref{thmsec3} to the power divergence statistics would result in a bound with far from optimal dependence on these values. Again, we expect this to be the case, because we obtain our bounds by applying a general bound.  A possible direct for future research would be to proceed in the spirit of \cite{gaunt chi square} and use an approach specifically targeted at the power divergence statistics to obtain a bound with a better dependence on these values.

\end{enumerate}
\end{remark}

\begin{remark}Unless $\lambda\in\mathbb{Z}^+$ or $\lambda\geq5$, the presence of a singularity at $w_k=-\sqrt{np_k}$ for a least one of the partial derivatives up to sixth order of $b$ means we cannot apply Theorem \ref{thmsec3} to $T_\lambda(\mathbf{W})$.  We do, however, conjecture that the $O(n^{-1})$ rate holds for smooth test functions for any $\lambda$.  This conjecture is based on the fact that the $O(n^{(r-1)/r})$ Kolmogorov distance bounds of \cite{asylbekov} and \cite{ulyanov} are valid for all $\lambda\in\mathbb{R}$ and also the $O(n^{-1/2})$ series formula (\ref{seriesexp}) for $R(\mathbf{w})$.  Extending the theory of this paper to deal with the whole family of power divergence is an interesting direct for future research.
\end{remark}

\section*{Acknowledgements}
The authors would like to thank Persi Diaconis for bringing this problem to our attention.  RG is supported by EPSRC grant EP/K032402/1.  GR acknowledges support from  EPSRC grants GR/R52183/01 and  EP/K032402/1.


\begin{thebibliography}{99}
\bibitem{pearson} Pearson, K.  On the criterion that a given system of deviations is such that it can be reasonably supposed to have arisen from random sampling.  \emph{Phil. Mag.} $\mathbf{50}$ (1900), pp. 157--175.
\end{thebibliography}
\end{document}